\documentclass[11pt,a4paper,reqno]{amsart}
\usepackage[a4paper,top=3cm,bottom=3cm,inner=3cm,outer=3cm,includehead,includefoot]{geometry}
\usepackage{amsmath,amssymb,amsthm,calc,verbatim,mathrsfs,paralist,tikz,url,cite}
\usetikzlibrary{shapes.misc,calc,intersections,patterns,decorations.pathreplacing}

\newtheorem{theorem}{Theorem}
\newtheorem{lemma}[theorem]{Lemma}

\newtheorem{definition}[theorem]{Definition}
\newtheorem{question}[theorem]{Question}

\def\E{\mathbb{E}}
\def\N{\mathbb{N}}
\def\P{\mathbb{P}}

\def\Z{\mathbb{Z}}
\newcommand{\bin}{\operatorname{Bin}}
\newcommand{\strongly}{G^S}
\newcommand{\good}{G}
\newcommand{\interior}{\operatorname{int}}
\newcommand{\given}{\, \big| \,}
\renewcommand{\lg}{\operatorname{long}}
\newcommand{\sh}{\operatorname{short}}
\newcommand{\crest}{\operatorname{Cr}}
\renewcommand{\epsilon}{\varepsilon}
\renewcommand{\leq}{\leqslant}
\renewcommand{\geq}{\geqslant}
\renewcommand{\to}{\rightarrow}

\title{The time of bootstrap percolation in two dimensions}

\author[P. Balister]{Paul Balister}
\address{Department of Mathematical Sciences, University of Memphis, Memphis, TN 38152, USA}
\email{pbalistr@memphis.edu}

\author[B. Bollob\'as]{B\'ela Bollob\'as}
\address{Department of Pure Mathematics and Mathematical Statistics, Wilberforce Road, Cambridge, CB3 0WA, UK, and Department of Mathematical Sciences, University of Memphis, Memphis, TN 38152, USA, and London Institute for Mathematical Sciences, 35a South Street, London, W1K 2XF, UK}
\email{b.bollobas@dpmms.cam.ac.uk}

\author[P.J. Smith]{Paul Smith}
\address{IMPA, 110 Estrada Dona Castorina, Jardim Bot\^{a}nico, Rio de Janeiro, 22460-320, Brazil}
\email{psmith@impa.br}

\date{\today}


\begin{document}

\begin{abstract}
We study the distribution of the percolation time $T$ of 2-neighbour bootstrap percolation on $[n]^2$ with initial set $A\sim\bin([n]^2,p)$. We determine $T$ up to a constant factor with high probability for all $p$ above the critical probability for percolation, and to within a $1+o(1)$ factor for a large range of $p$.
\end{abstract}

\maketitle

\section{Introduction}

The subject of this paper is bootstrap percolation, a type of two-state cellular automaton introduced by Chalupa, Leath and Reich in 1979 \cite{CLR} to model certain interacting particle systems in physics. In \emph{$r$-neighbour bootstrap percolation} on a graph $G$, vertices are either \emph{infected} or \emph{uninfected}, and the states of vertices evolve at discrete times according to the following process. At time $t=0$, there is an initial set $A\subset V(G)$ of infected vertices, and all other vertices in the graph are uninfected. Thereafter, at each discrete time, uninfected vertices become infected if at least $r$ of their neighbours are already infected, while infected vertices remain infected forever. Thus, we set $A_0=A$, and for each integer $t\geq 1$, the set of infected vertices at time $t$ is
\[
A_t := A_{t-1} \cup \big\{ v\in V(G): |N(v)\cap A_{t-1}| \geq r \big\},
\]
where $N(v)$ denotes the set of neighbouring vertices of $v$ in $G$. The graph $G$ is often taken to be $\Z^d$ or $[n]^d=\{1,\dots,n\}^d$, where in both cases edges are between vertices which differ by exactly $1$ in exactly one coordinate. We write $[A]=\cup_{t=0}^\infty A_t$ and call $[A]$ the \emph{closure} or \emph{span} of $A$. We say \emph{$A$ percolates $G$} if $[A]=V(G)$. Occasionally we use the notation $[X]_t$ to mean the set of infected vertices at time $t$ when the initial set is $X$. A subset $U$ of $V(G)$ is said to be \emph{internally spanned} if $U\subset[A\cap U]$.

Bootstrap percolation may be thought of as a monotone version of the Glauber dynamics of the Ising model, and it is here that many of its applications lie. For example, Fontes, Schonmann and Sidoravicius \cite{FSSIsing} and Morris \cite{MorrisGlauber} used results from bootstrap percolation to prove bounds on the critical threshold for fixation at the Gibbs state in the Ising model. Bootstrap percolation has also found applications in crack formation, clustering phenomena, dynamics of glasses \cite{GST}, sandpiles \cite{FLP}, jamming \cite{DGLBD}, and many other areas of statistical physics \cite{Adler,AdLev,ASA}.

Many of the most widely studied questions in bootstrap percolation ask what one can say about the properties of the system when the initial set is chosen randomly. By ``randomly'' here we mean that each vertex of $V(G)$ is included in $A$ independently with probability $p$; sometimes we say that $A$ is a \emph{$p$-random} subset of $V(G)$, and we write $A\sim\bin(V(G),p)$. One would like to know how likely percolation is to occur in this setting, as a function of the graph $G$, the infection parameter $r$, and the initial infection probability $p$. In the case of $r$-neighbour bootstrap percolation on the lattice graph $[n]^d$, with $d$ fixed and $n$ tending to infinity, it is known that there exists a sharp phase transition for percolation for all $2\leq r\leq d$. This means that there is a function $p_c=p_c([n]^d,r)$ such that for all $\epsilon>0$, if $p\leq(1-\epsilon)p_c$ then with high probability there is no percolation, while if $p\geq(1+\epsilon)p_c$ then with high probability there is percolation. The function $p_c$ is known as the \emph{critical probability} for percolation. A certain weaker form of this result was proved by Aizenman and Lebowitz \cite{AL} in 1988 for $r=2$, in a paper that started the study of the critical probability on finite graphs. The analogous results for $r\geq 3$ were proved considerably later: by Cerf and Cirillo \cite{CerfCir} for $r=3$ and Cerf and Manzo \cite{CerfManzo} for $r\geq 4$. The sharper form we have just stated has a similar history: in 2002, Holroyd \cite{Hol} proved that $p_c([n]^2,2)=(\pi^2/18+o(1))/\log n$; in 2009, Balogh, Bollob\'as and Morris \cite{BBM3D} established the sharp threshold for $r=3$; and the full result was proved by Balogh, Bollob\'as, Duminil-Copin and Morris \cite{BBDCM} in 2012. Sharp thresholds are also known to exist for several other bootstrap models, including the hypercube \cite{BBhyp,BBMhigh} and a number of other models on $\Z^2$ \cite{DCH,DCvE}. Moreover, recent work of Bollob\'as, Smith and Uzzell \cite{BSUgen} and Bollob\'as, Duminil-Copin, Morris and Smith \cite{BDCMS} shows that similar threshold behaviour, albeit in a weaker sense, is exhibited by a considerably larger class of two-dimensional bootstrap percolation processes. %In some cases the phase transition is much sharper than stated above: among other results, Balogh and Bollob\'as \cite{BBsharp} proved that for $d=r=2$ the fixed $\epsilon$ above can be replaced by any function $\epsilon (n)>0$ with $\epsilon (n)(\log n)^2/\log\log n=\Omega(1)$.

Given a graph $G$ and an initial infection probability $p$ such that percolation is likely to occur, one would also like to know how long percolation takes. Thus, letting $T$ denote the random variable $\min\{t:A_t=V(G)\}$, which we call the \emph{percolation time} of the set $A$, the question is to determine information about the distribution of $T$. In particular, how concentrated is $T$?

Before continuing, we mention that the study of \emph{typical} infection times on infinite lattices (our objects of study -- total infection times -- are quite different) has an extensive history (for just a small selection, see \cite{AL,Hol,BBDCM,AMS}) and these are by now well understood in a variety of settings. In particular, in all cases that have so far been studied, there has been shown to exist a close inverse relationship between typical occupation times on infinite lattices and critical probabilities for complete infection on the corresponding finite lattices. To illustrate this phenomenon, in the 2-neighbour model on $\Z^2$ with a $p$-random initial configuration, let $\tau$ be the minimal $t$ such that the origin is infected by time $t$. Then one result of \cite{Hol} is that $\tau = p_c^{-1}\big(p(1+o(1)\big)$ with high probability, where $p_c=p_c([n]^2,2)$ is the critical probability for percolation on $[n]^2$ as a function of $n$, and $p_c^{-1}$ is its inverse function. We emphasize again, however, that our object of study is quite different: we are interested in determining information about the \emph{total} infection time, rather than the \emph{typical} occupation time, and there does not seem to be a straightforward relationship between the two. Indeed, we do not use any results or techniques specifically related to the latter in our study of the former.

All known proofs of bounds for the critical probability in the various bootstrap percolation processes also give some (rather limited) information about the percolation time, although the bounds one can extract are never explicitly stated in these papers. (Of course, this is not surprising: the papers are not concerned with studying the percolation time.) For example, the methods in \cite{AL} and \cite{Hol} for proving that percolation is likely to occur in 2-neighbour bootstrap percolation on $[n]^2$ show that if $p\geq(1+\epsilon)p_c([n]^2,2)$ then $T\leq n(\log n)^{2+o(1)}$ with high probability as $n$ tends to infinity. (Actually, a simple adaptation of the proof of this statement can be used to show under the same conditions that the percolation time satisfies the stronger inequality $T=O(n\log n)$ with high probability, and we use this adaptation in the proofs of both main theorems in this paper.) From below, the bounds one can extract are even weaker: for example, again in the 2-neighbour model on $[n]^2$, all one can deduce from \cite{Hol} is that if $p=(1+\epsilon)p_c([n]^2,2)$ then with high probability $T\geq\E[\tau] = n^{1-o(1)}$ as $\epsilon\to 0$ (where $\tau$ is the occupation time of the origin (say) in $\Z^2$, as in the previous paragraph).

The only known sharp results about the time of percolation relate to the $r$-neighbour model on the torus $(\Z/n\Z)^d$ when $p$ is close to $1$. With such a large initial infection probability, and therefore such a small percolation time, one might expect the events that sites in $(\Z/n\Z)^d$ are uninfected at time $t$ to be approximately independent, and therefore the number of uninfected sites at time $t$ to be approximately Poisson distributed. Bollob\'as, Holmgren, Smith and Uzzell \cite{BHSU} ($d$-neighbour in $d$ dimensions) and Bollob\'as, Smith and Uzzell \cite{BSUr} ($r$-neighbour in $d$ dimensions) make this heuristic precise using extremal techniques and the Stein-Chen method. They show that if $p$ satisfies certain conditions depending on $t$ and $n$ (which in particular imply that $p=1-o(1)$), then with high probability the percolation time is exactly equal to $t$, or in some cases to either $t$ or $t+1$. A weaker statement, which follows from Theorem $1.3$ of \cite{BHSU} (after observing that the expression $m_t$ in that theorem is asymptotically $4t$ as $t\to\infty$), says that if $r=d=2$ and $\log\log n \ll \log 1/(1-p) \ll \log n$ then
\begin{equation}\label{eq:Tlargep}
T = \frac{(1+o(1))\log n}{2\log 1/(1-p)}
\end{equation}
with high probability. (We use the notation $f(n)\ll g(n)$ to mean that $g(n)/f(n)\to\infty$ as $n\to\infty$.) The condition $\log\log n \ll \log 1/(1-p)$ above corresponds to the upper bound on the time $t$ in the statement of Theorem $1.3$ in \cite{BHSU}, which is the natural limit of the Stein-Chen method. The condition $\log 1/(1-p)\ll\log n$ is required to ensure that the expression for $T$ in \eqref{eq:Tlargep} tends to infinity. For larger $p$ (that is, $p$ such that $\log n / \log 1/(1-p)$ is bounded), the theorems in \cite{BHSU,BSUr} instead show (in many cases) that the percolation time is equal to a constant with high probability as $n\to\infty$. Indeed, the \emph{aim} of \cite{BHSU,BSUr} was to establish one- and two-point concentration for the percolation time.

The first of the two main theorems in this paper says that the expression \eqref{eq:Tlargep} for the percolation time holds for a much broader range of sequences of initial infection probabilities: not only do we drop the condition $p=1-o(1)$, but in fact we only require $\liminf p\log\log n > \pi^2/9$. It is worth remarking that while the techniques of \cite{BHSU,BSUr} are broadly similar to each other, exploiting the near-independence of the states of sites (a property that we mentioned several paragraphs ago), the techniques of the present paper are fundamentally different. These new proof techniques, which we sketch in Section \ref{se:sketch}, together form one reason that we are able to relax the lower bound on $p$ quite so much compared with \cite{BHSU,BSUr}.

Throughout this paper we use $T$ to denote the percolation time of a $p$-random subset of $[n]^2$ under the 2-neighbour bootstrap percolation process. We also fix the constant $\lambda=\pi^2/18$; the reader may recall this constant appearing in the result of Holroyd \cite{Hol} on the critical probability $p_c([n]^2,2)$.
% the significance of which will be made clear shortly.
The following is our first theorem.

\begin{theorem}\label{th:large}
Let $0< p=p(n)<1 $ be such that $\liminf p\log\log n > 2\lambda$ and $1-p=n^{-o(1)}$ (that is, $\log 1/(1-p)\ll\log n$). Then
\[
T = \frac{(1+o(1))\log n}{2\log 1/(1-p)}
\]
with high probability as $n\to\infty$.
\end{theorem}

A natural example of an event that would prevent percolation happening by time $t$ is the existence of an empty $(2t+1)\times 2$ rectangle in the initial set $A$. (Such a rectangle with a site missing at either end would also suffice, but since we are only interested in determining $T$ asymptotically, we do not need to be that precise.) One can easily show that the largest $t$ for which such a rectangle is likely to exist is about $(\log n) / \big(2\log 1/(1-p)\big)$. This observation essentially proves the lower bound of Theorem \ref{th:large}; the real content of the theorem is therefore the upper bound.

A detailed sketch of the proof of Theorem \ref{th:large} is given in Section \ref{se:sketch}. Here we mention only briefly that the proof proceeds by showing that the event that a site $x$ is uninfected at time $t$ implies (deterministically) the existence of a ``path'' of $L\times L$ squares near to $x$, each of which is not internally spanned. For an appropriately optimized value of $L$ (and hence length of the ``path''), the probability of the latter event can be shown to be small. This summary of the proof is so brief as to be positively misleading, however: for example, the event that the $L\times L$ squares are internally spanned turns out to be too strong (in the sense that it is not quite necessary for the event whose probability we are trying to bound, and in fact its probability is vastly different), and so a weakening of this is required, which leads to considerable complications. Once again, we refer the reader to Section \ref{se:sketch} for a more comprehensive outline of the proof.

Our second theorem establishes the percolation time up to a constant factor with high probability in the remaining case, when $p$ is supercritical but small. When $p$ is in the range of Theorem \ref{th:large} (we shall call this the ``large $p$ regime''), the percolation time is with high probability asymptotically equal to one half of the length of the longest initially empty double row or column. When $p$ is in the range of Theorem \ref{th:small}, the percolation time is with high probability much larger than the length of the longest initially empty double row or column. We shall call this  the ``small $p$ regime'', although sometimes we shall reserve this phrase for the special case when $\liminf p\log\log n$ is strictly less than $2\lambda$.

\begin{theorem}\label{th:small}
There exists a function $\mu:(0,1)\to(0,\infty)$, with $\mu(p)=\lambda+o(1)$ as $p\to 0$, such that the following holds. Let $p=p(n)$ be such that $\liminf p\log n>\lambda$ and $p=o(1)$. Then
\begin{equation}\label{eq:small}
T = \Theta\big( t(n,p) \big)
\end{equation}
with high probability as $n\to\infty$, where
\begin{equation}\label{eq:small2}
t(n,p) := \max\left\{ \sqrt{\frac{\log n - \mu(p)/p}{p}} \exp\left(\frac{\mu(p)}{p}\right) \, , \, \frac{\log n}{p} \right\}.
\end{equation}
\end{theorem}

Let $t_1(n,p)$ denote the first of the two functions inside the maximum in \eqref{eq:small2} and let $t_2(n,p)=(\log n)/p$ denote the second. If $\limsup p\log\log n<2\lambda$ then $t_1(n,p)\gg t_2(n,p)$, while if $\liminf p\log\log n>2\lambda$ then $t_1(n,p)\ll t_2(n,p)$. Thus, as $p$ becomes small, the point at which $t_1(n,p)$ starts to become larger than $t_2(n,p)$ (and thus $T=\Theta\big(t_1(n,p)\big)$ wth high probability) occurs precisely at the point at which the conditions for Theorem \ref{th:large} fail. Thus, for almost the entire range of $p$ for which Theorem \ref{th:small} applies but Theorem \ref{th:large} does not, the theorem says that $T=\Theta\big(t_1(n,p)\big)$ with high probability, and therefore this result is the main content of Theorem \ref{th:small}. However, at the transition itself, when $p=2\lambda/\log\log n$, it is not possible to say which of the two functions is larger without knowing more about the function $\mu(p)$, so it is not possible to omit the function $t_2(n,p)$ from Theorem \ref{th:small}.

The nature of the $o(1)$ term in the function $\mu(p)$ is dependent on the second and higher order terms in the critical probability $p_c([n]^2,2)$. The precise definition of $\mu(p)$ is given in Lemma \ref{le:mu}, but roughly speaking it is such that the probability that a grid of side length $\exp\big(\mu(p)/p\big)$ contains an internally spanned ``critical droplet'' (that is, a droplet of side length approximately a power of $1/p$; a precise definition is given in Section \ref{se:defs}) is equal to $1/2+o(1)$. (The definition in Lemma \ref{le:mu} includes an extra condition, which is needed for technical reasons.) Unfortunately, even with the recent result of Morris \cite{MorrisSharp} identifying the second order term in $p_c([n]^2,2)$ up to a constant factor, it is only possible to say that $|\mu(p)-\lambda|$ is at most $O(\sqrt{p})$. Thus, since $e^{c/\sqrt{p}}\gg \sqrt{(\log n)/p}$ for small enough $p$ and constant $c$, the main feature of Theorem \ref{th:small} is the assertion that there \emph{exists} a function $t(n,p)$ such that \eqref{eq:small} holds, not the formula for $t(n,p)$ in \eqref{eq:small2}. (In particular, when $p=(1+\epsilon)p_c([n]^2,2)$, all that can be said about the function $t(n,p)$ in Theorem \ref{th:small} is that it is equal to $n^{1-o(1)}$ as $\epsilon\to 0$. Of course, this is not the point of the theorem.)

Holroyd \cite{Hol} proved that the condition $\liminf p\log n>\lambda$ ensures that the initial set percolates with high probability, and that the condition $\limsup p\log n<\lambda$ ensures that with high probability the initial set does not percolate. It is natural to ask whether the conclusion of Theorem \ref{th:small} holds conditioned only on percolation occurring, dropping the assumption that $\liminf p\log n>\lambda$. However, this is not the case. When $p\approx\lambda/\log n$, the probability of percolation is roughly constant and the number of critical droplets is approximately Poisson distributed. Thus, if percolation does occur, then the percolation time will depend on the number of critical droplets and their relative positions.

As for Theorem \ref{th:large}, a detailed sketch of the proof of Theorem \ref{th:small} is given in Section \ref{se:sketch}, and so here we describe only the fundamental approach. The idea is, for $\liminf p/p_c([n]^2,2) > 1$, to find a maximal area of the grid not containing an internally spanned critical droplet (we call this a ``sparse'' region), and to bound from below the time it takes for this region to fill, assuming that every site outside the region is initially infected. The percolation time for the whole grid is then clearly at least this time. In order to obtain the required bound on the time for the sparse region to fill, we undertake a detailed study of the geometry of such regions of the grid. As far as the authors are aware, this work represents the first systematic study of the geometry of regions of the grid \emph{not} containing internally spanned critical droplets.

In the next section we sketch some of the main ideas that go into our proofs of Theorems \ref{th:large} and \ref{th:small}. In Section \ref{se:defs} we recall some standard notation and lemmas from bootstrap percolation and we introduce some new notation specifically related to the percolation time. In Section \ref{se:critgrids} we make formal the notion and properties of a ``critical grid size'', which is a function $K=K(p)$ such that the probability a $p$-random subset of $[K]^2$ percolates is approximately constant. This may be thought of as an inverse to the problem of determining the critical probability, which is a function $p_c=p_c(n)$ such that the probability a $p_c$-random subset of $[n]^2$ percolates is approximately constant. The proof of Theorem \ref{th:large} including the method of tiling with $L$-cells is then given in Section \ref{se:large}, and finally Sections \ref{se:uppersmall} and \ref{se:lowersmall} contain the proofs of the upper and lower bounds of Theorem \ref{th:small} respectively.

\section{Sketch of proofs}\label{se:sketch}

We now sketch some of the most important details of the proofs of Theorems~\ref{th:large} and~\ref{th:small}. We emphasize again that there is essentially no overlap between the proofs in the present paper and those from earlier works on the time of bootstrap percolation \cite{BHSU,BSUr}. Here, our only tools are basic properties of the 2-neighbour bootstrap percolation model on $[n]^2$ (such as the ``rectangles process'' and the notion of a ``critical droplet'', which are explained below), results of Holroyd \cite{Hol} on $p_c([n]^2,2)$, and a lemma from graph theory (Lemma \ref{le:coffeetime}).

\subsection{Sketch of the proof of Theorem \ref{th:large}}

As mentioned in the introduction, the lower bound of Theorem \ref{th:large} is no more than the trivial assertion that there exists a $(2t+1) \times 2$ rectangle that does not contain any initially infected sites, where
\[
t = \frac{(1-\epsilon)\log n}{2 \log 1/(1-p)}
\]
and $\epsilon>0$. The real interest in Theorem \ref{th:large} therefore lies in the upper bound, which may be thought of as saying that, in this range of the initial infection probability $p$, initially uninfected $(2t+1)\times 2$ blocks of sites are the only obstacles to percolation by time $t$.

%From above, one would like to show that the existence of an empty $(2t-1)\times 2$ rectangle in the initial set $A$ is a necessary condition for the percolation time to be at least $t$ (with high probability, that is; the statement is obviously not true deterministically).

Suppose a site $x$ in $[n]^2$ is uninfected at time $t$. It is easy to see that $x$ must be contained in a $2\times 2$ square of uninfected sites at time $t-2$. In fact, provided $x$ is not too close to the boundary of $[n]^2$, it is easy to see that there must exist a sequence of $t-1$ initially uninfected sites, starting with the top-right site in the $2\times 2$ square, and continuing either up or right each time, and that a similar statement, with the correct mix of up/down and left/right, also holds for the three other sites in the $2\times 2$ square. We would like to show that by far the most probable way for this to occur is for these four paths to be aligned to form a $(2t-2)\times 2$ rectangle, or more specifically, we would like to show that the probability the four uninfected paths exist is not much more than $(1-p)^{4t-4}$, which is just the probability that a given $(2t-2)\times 2$ rectangle is initially empty.

A first attempt at a proof might go as follows. Assume that all four paths of uninfected sites start by growing out horizontally from the $2\times 2$ square, so that they form an uninfected rectangle of height $2$ and unknown length. Let us concentrate on the top-right path, which we call $P$. If the path ever strays away from the horizontal line it starts along, then that should be at the cost of many new uninfected sites, because a path of uninfected sites that contains corners is not closed. The trouble is that there are too many choices for the paths, so the cost of this gain in probability is a large combinatorial factor.

However, it is possible to salvage this attempt at a proof. Rather than counting top-right paths of sites individually, we look at the intersection of top-right paths with a much coarser grid of squares, of a certain side length $L$, and count these. First we allow an initial time $t'=BL/p=o(t)$, where $B$ is a constant. By this time we expect nearly all internally spanned squares of side length $L$ --- which we call \emph{$L$-cells} --- to have filled. Now consider just the first $t-t'$ sites in the top-right path $P$: at time $t'$ they are still uninfected, and they intersect a path of $L$-cells all of which are either not internally spanned or slow to fill; we call such $L$-cells \emph{bad}. There is now an optimization question: how large should $L$ be to minimize the probability of this event, that there is a path of bad $L$-cells? In order to have any hope of this method working, the probability that an $L$-cell is bad should be at most $(1-p)^{(1+c)L}$, for some $c>0$. This is because we would like to show that the probability there exists a path of bad $L$-cells is about $(1-p)^t$, so we need the additional $c$ to overcome the combinatorial factor that comes from taking a union bound over all paths. Thus, $L$ must be large enough for the probability that an $L$-cell is bad to be small. Another reason $L$ should be large is to minimize the combinatorial factor. As $L$ increases, there are fewer paths of $L$-cells inside a square of side length $t-t'$, so the combinatorial factor decreases. On the other hand, $L$ cannot be too large, because the error time $t'=BL/p$ must be $o(t)$. The $L$ that we choose is the smallest $L$ such that the probability an $L$-cell is bad is approximately $(1-p)^{(1+c)L}$. (In fact we take $c=1$.)

This second attempt at a proof is also not quite right: the probability that an $L$-cell is bad, as we have defined it, is at least $(1-p)^L$ because if any of the four edges of the square is empty then the square cannot be internally spanned. On the other hand we have said that the probability needs to be at most $(1-p)^{(1+c)L}$, so our definition of bad cannot be the right one. The way around this is as follows. One can show that, at the scale we are considering, empty edges of the $L$-cell are the only first order obstructions to being internally spanned, and that by strengthening the definition of bad so that an $L$-cell is only bad if it is not internally spanned except possibly for one or more of its edges, then the probability that an $L$-cell is bad now correlates with $(1-p)^{2L}$. While this gives the desired probability bound, it is no longer true that the original path $P$ of uninfected sites intersects a long path of bad $L$-cells, because $P$ may intersect only the edges of one or more of the $L$-cells. However, these paths are so restricted that they contribute little combinatorially to the union bound.

\subsection{Sketch of the proof of Theorem \ref{th:small}}

Blocking sets in the large $p$ regime are just (approximately) empty $(2t+1)\times t$ rectangles in the initial set $A$. In the small $p$ regime, blocking sets are much less straightforward. Loosely speaking, they are large, sparse regions of $A$. Before we say what we mean by ``sparse'' (and ``large''), we need to introduce the notion of a critical droplet. In bootstrap percolation on $[n]^2$ (and similar statements hold for other lattice grids in other dimensions) it is known that there is a threshold length, roughly at a power of $\log n$, such that the existence of an internally spanned rectangle with perimeter at least this length is enough to ensure percolation of the whole grid with high probability. Rectangles of this perimeter are known as \emph{critical droplets}. (There is a formal definition, which we give in the next section.) The sparse regions of $A$ that act as blocking sets in the small $p$ regime are maximal regions of the grid not containing an internally spanned critical droplet.

There are two parts to the proof of the lower bound in Theorem \ref{th:small}. First, we determine the size and shape of these maximal sparse regions. For this we use many of the same tools as we use in the large $p$ regime. Second, we show that the sparse regions percolate slowly, even under the additional assumption that the rest of the grid is initially full. The principal technical difficulty lies in showing that the spread of infection through the sparse regions occurs at the speed one would expect. This is the main result of Sections \ref{se:waves}, \ref{se:restrictions} and \ref{se:slowperc}. If the sparse region is infected quickly then we may ask how the information travelled from the edge of the sparse region to the centre. We show that there must exist a sequence of internally spanned rectangles located much closer together than one would expect, and that, in a certain sense, these droplets join the edge of the sparse region to the centre. Such a sequence of rectangles, which is defined formally in Section \ref{se:waves}, is called a \emph{wave}. We bound the number of waves in terms of the size of the sparse region and the time it takes the region to become infected, and we also bound the probability that any given wave exists. A more detailed sketch of the proof is given at the beginning of Section \ref{se:lowersmall}.

The upper bound of Theorem \ref{th:small} is the easier of the two bounds, and is proved in Section \ref{se:uppersmall}. For the upper bound in the large $p$ regime we focus on squares of side length $L$. In the small $p$ regime we do something similar, although we work at a different scale $M$, which is related to the quantity on the right-hand side of \eqref{eq:small}. We tile the grid $[n]^2$ with $M$-cells and wait an initial time $BM/p$, where again $B$ is a constant. As in the large $p$ regime, we expect most $M$-cells to have internally spanned by this time, and we call those that have not \emph{weakly bad} (\emph{weakly} here emphasizes that the property is weaker than that of being bad, because, unlike in the large $p$ regime, we only require that the whole cell, including its edges, is not internally spanned by time $BM/p$). The proof then uses a graph theoretic lemma that bounds the number of order $k$ connected induced subgraphs of a graph $G$ containing a specific vertex in terms of $k$ and the maximum degree of $G$. This lemma allows one to say that the largest connected component of weakly bad $M$-cells is not too large --- in fact, the total area of the component is likely to be equal (to within a constant factor) to the area of the largest sparse region of the grid, where, as before, sparse means ``not containing an internally spanned critical droplet''. Finally we observe that any component of weakly bad $M$-cells is infected by the surrounding cells in time proportional to its size.

\section{Definitions and tools}\label{se:defs}

The first few definitions we need are used throughout the bootstrap percolation literature. Recall that a set $X\subset[n]^2$ is \emph{internally spanned} if $X\subset[X\cap A]$, where $A$ is (as always) the initial set. The set $X$ is \emph{empty} if $A\cap X=\emptyset$, it is \emph{occupied} if $A\cap X\neq\emptyset$, and it is \emph{full} if $A\cap X=X$. A \emph{droplet} is a rectangular subset of $[n]^2$ of the form
\[
D = [(a,b),(c,d)] := \big\{ (x,y)\in\Z^2 \, : \, a\leq x\leq c, \, b\leq y\leq d \big\}.
\]
The \emph{dimensions} of $D$ are $\dim(D)=(c-a+1,d-b+1)$, the \emph{long} and \emph{short side-lengths} of $D$ are respectively $\lg(D)=\max\{c-a+1,d-b+1\}$ and $\sh(D)=\min\{c-a+1,d-b+1\}$, and the \emph{semi-perimeter} of $D$ is $\phi(D)=\lg(D)+\sh(D)$. An \emph{$m$-cell} is a droplet $D$ with $\lg(D)=\sh(D)=m$. The \emph{interior} of an $m$-cell $D=[(a,b),(c,d)]$ is the $(m-2)$-cell $\interior(D)=[(a+1,b+1),(c-1,d-1)]$ and the \emph{edge} of $D$ is the set $\partial D = D\setminus\interior(D)$. The \emph{left edge} of $D$ is the set $[(a,b),(a,d)]$, and the \emph{right}, \emph{top} and \emph{bottom} edges of $D$ are defined similarly.

The concept of a critical droplet was mentioned briefly in the introduction, in the context of blocking sets in the small $p$ regime. Here we make that notion precise. Let $\gamma(p)=p^{-3}$. A \emph{critical droplet} is a droplet $D$ for which $\gamma(p)/2\leq\phi(D)\leq\gamma(p)$. The event that a set $X\subset[n]^2$ contains an internally spanned critical droplet is written $\Gamma(X)$. For brevity, we shall usually write $\gamma$ for $\gamma(p)$ and $\Gamma(n)$ for $\Gamma([n]^2)$.

The next few definitions relate specifically to the time of percolation. The event that the set $X$ is internally spanned is written $I(X)$. The event that $[X]_t=X$ (that is, that the set $X$ is internally spanned by time $t$) is denoted $I_t(X)$. The $m$-cell $D$ is \emph{strongly good} if it is internally spanned by time $Bm/p$, where $B$ is a sufficiently large absolute constant and $A\sim\bin(D,p)$ is the initial set. Thus, $D$ is strongly good if $I_{Bm/p}(D)$ occurs. It is \emph{good} if its span by time $Bm/p$ contains $\interior(D)$. Formally, $D$ is good if $\interior(D)\subset[D\cap A]_{Bm/p}$. Finally, $D$ is \emph{semi-good} if it is good but not strongly good, \emph{weakly bad} if it is not strongly good, and \emph{bad} if it is not good. We write $\strongly(D)$ for the event that $D$ is strongly good and $\good(D)$ for the event that $D$ is good. We also use $\strongly$ and $\good$ for the associated indicator functions. Let $\eta_m$ be the probability that an $m$-cell is bad and $\theta_m$ the probability that an $m$-cell is weakly bad; thus, for an $m$-cell $D$,
\[
\eta_m = \P_p\big(\good(D)^c\big) \quad \text{and} \quad \theta_m = \P_p\big(\strongly(D)^c\big).
\]

One of the fundamental tools in the study of bootstrap percolation is the \emph{rectangles process}, an algorithm which exactly describes the evolution of the 2-neighbour bootstrap process on $[n]^d$, but in a way which does not preserve infection times of sites. The algorithm was first observed by Aizenman and Lebowitz (\cite{AL}, Lemma 1), who used it to prove a lower bound for the critical probability of 2-neighbour bootstrap percolation on $[n]^d$. The algorithm runs as follows. First, consider each initially infected site to be a droplet with dimensions $(1,1)$. Then repeat the following process: whenever there are two droplets $D_1$ and $D_2$ and sites $x_1\in D_1$ and $x_2\in D_2$ with $\|x_1-x_2\|_1\leq 2$, replace $D_1$ and $D_2$ by the smallest droplet containing both. (Observe that $D_1$ and $D_2$ need not be disjoint, and they may even be nested.) If two such droplets do not exist, stop the algorithm. The set of sites contained in the final configuration of rectangles is precisely the closure of the initial set.

It may seem strange that an algorithm which is not able to encode the times at which sites become infected should be useful for proving results about the time of percolation, but its importance lies in the following lemma, due to Aizenman and Lebowitz. The lemma says that if a droplet is internally spanned then it must also contain internally spanned droplets at all smaller scales. 

\begin{lemma}\label{le:AL}
Let $D$ be an internally spanned droplet. Then for all $1\leq k\leq\lg(D)/2$ there exists an internally spanned droplet $D'\subset D$ such that $k\leq\lg(D')\leq 2k$.
\end{lemma}

The proof is immediate from the algorithm: if $D$ is the smallest droplet containing $D_1$ and $D_2$, and if $D_1$ and $D_2$ are close enough to be merged in the rectangles process, then it is easy to see that $\lg(D)\leq\lg(D_1)+\lg(D_2)$+1.

Another immediate and important consequence of the rectangles process (although there are many other ways of proving it) is the following lemma, the last of this section. A proof can be found in \cite[pp.~104--105]{CiM}.

\begin{lemma}\label{le:phi}
Let $D$ be a droplet internally spanned by a set $A$. Then $|A|\geq\phi(D)/2$.\qed
\end{lemma}

Finally, on a matter of notation, we remark that throughout the paper $c$ and $C$ will always denote absolute positive constants. To avoid accumulating notation we shall frequently reuse both $c$ and $C$ to mean different positive constants, occasionally even doing so inside a proof.

\section{Critical grid sizes and the inverse of the critical probability}\label{se:critgrids}

The problem of finding the critical probability for 2-neighbour bootstrap percolation on $[n]^2$ (and similarly for other models of bootstrap percolation) can be thought of as that of finding $p$ as a function of $n$ (which is assumed to be large) such that a $p$-random initial subset of $[n]^2$ has approximately a constant probability of percolating. In this paper we make extensive use of pairs $(n,p)$ with this property, but here we require $n$ to be a function of $p$, rather than the other way around. Thus our problem is essentially that of finding the inverse of the critical probability, which we think of as the ``critical grid size''. That sounds easy enough, but for our applications we require a slightly stronger property than a constant probability of percolating: we require $n$ as function of $p$ (which need not be small -- this is another small technicality) such that a $p$-random initial subset $A$ of $[n]$ has the following two properties. First, that the probability $[n]^2$ is strongly good (that is, that $A$ spans $[n]^2$ in time at most $Bn/p$) is at least a small positive constant. Second, that the probability $A$ contains an internally spanned critical droplet is at most a slightly larger positive constant. On the surface these two properties seem to be very different, so it is not obvious that such an $n$ should exist.

Our first lemma shows that if $n$ is sufficiently large and $p$ is such that a $p$-random initial subset of $[n]^2$ contains an internally spanned critical droplet with probability at least a constant, then with probability only a very slightly smaller, $[n]^2$ is strongly good (with initial set $A$). The proof is a minor adaptation of the deduction of Theorem 1 (i) from Theorem 2 (i) in \cite{Hol}.

\begin{lemma}\label{le:holroydplus}
Let $\alpha,p,\epsilon\in(0,1)$ and let $n\in\N$ be sufficiently large. Suppose that $\P_p\big(\Gamma(n)\big)\geq\alpha$. Then $\P_p\big(I_{6n/p}([n]^2)\big)\geq(1-\epsilon)\alpha$.
\end{lemma}

\begin{proof}
If $\gamma/2<3p^{-1}\log n$ then $p$ is so large that the probability $[\sqrt{n}]^2$ is not internally spanned is $o(1/n)$. Hence, we can tile the grid with squares of side length $\sqrt{n}$ and with high probability they will all be internally spanned. Each such square takes time at most $n$ to fill, so $T\leq n$ with high probability. From now on we shall assume that $\gamma/2\geq 3p^{-1}\log n$.

Let $E$ be the event that every row or column of length $\gamma/2\geq 3p^{-1}\log n$ is non-empty. Thus
\begin{equation}\label{eq:probEc}
\P_p(E^c) \leq n^2 (1-p)^{3p^{-1}\log n} \leq \exp\big(2\log n - 3\log n\big) = o(1).
\end{equation}

Provided there exists an internally spanned critical droplet, the event $E$ ensures that $[n]^2$ is internally spanned. However, the proof only shows that the percolation time is $O(n\gamma)$. We introduce an additional event $F$ to ensure that the spread of infection from the critical droplet to the rest of the grid is fast, so that the percolation time is at most $6n/p$. First let $X_1(x,y)$ be the least $i\geq 0$ such that $(x+i,y)$ belongs to $A$. (It is convenient here to extend $A$ to a $p$-random subset of $\Z^2$ and to allow $x+i>n$. The intersection of the event $F$ with the event $E$ will imply that $x+i\leq n$ in all relevant cases.) Similarly, let $X_2(x,y)$ be the least $i\geq 0$ such that $(x,y+i)\in A$ (and again we allow $y+i>n$). Let
\[
X_1(x) = 2\sum_{y=1}^n X_1(x,y) \qquad \text{and} \qquad X_2(y) = 2\sum_{x=1}^n X_2(x,y).
\]
The purpose of defining $X_1(x)$ and $X_2(y)$ in this way is that if, for example, $[(x,1),(n,1)]$ is full and $X_1(x,y)\leq n-x$ for every $y$, then $X_1(x)+n$ is an upper bound for the time it takes $[(x,1),(n,n)]$ to become infected. To see this, observe that $[(x,2),(X_1(x,3),2)]$ is fully infected by time
\[
\max\big\{X_1(x,2),X_1(x,3)\big\} \leq X_1(x,2) + X_1(x,3),
\]
and inductively that $[(x,k),(X_1(x,k+1),k)]$ is fully infected by an additional time
\[
\max\big\{X_1(x,k),X_1(x,k+1)\big\} \leq X_1(x,k) + X_1(x,k+1)
\]
for all $2\leq k\leq n-1$. It then takes at most another $n$ steps for the rest of $[(x,0),(n,n)]$ to become infected.

Now we define the event $F$ to be that $X_1(x)\leq 4n/p$ for every $1\leq x\leq n-p^{-3}$ and that $X_2(y)\leq 4n/p$ for every $1\leq y\leq n-p^{-3}$. Observe that
\[
\P_p\left(\frac{X_1(x)}{2} > \frac{2n}{p}\right) = \P_p\left(\bin\left(\frac{2n}{p},p\right)<n\right),
\]
and by standard Chernoff bounds (for example, Theorem A.1.18 of \cite{ProbMeth}), this is at most $e^{-n/4}$. Hence $\P_p(F^c) \leq n^2 e^{-n/4}$. By combining this with \eqref{eq:probEc} and the assumption that $\P_p\big(\Gamma(n)\big)\geq\alpha$, the probability that $\Gamma(n)\cap E\cap F$ fails is
\begin{align*}
\P_p\big((\Gamma(n)\cap E\cap F)^c\big) &\leq \P_p\big(\Gamma(n)^c\big) + \P_p(E^c) + \P_p(F^c) \\
&\leq 1 - \alpha + n^{-1} + n^2e^{-n/4} \\
&\leq 1 - (1-\epsilon)\alpha,
\end{align*}
provided $n$ is sufficiently large.

Finally, $E$ ensures that $X_i(x,y)\leq p^{-3}$ for $i=1,2$ and for every $x$ and $y$, so given that $\Gamma(n)\cap E\cap F$ occurs, the percolation time is at most
\[
p^{-6} + 4np^{-1} + n \leq 6n/p.
\]
Here we have used the fact that $\alpha>0$ implies $p\geq(1-\delta)p_c([n]^2,2)$ for some $\delta>0$, and hence $p^{-6} \ll n/p$.
\end{proof}

In the next lemma we determine the critical grid size as a function of $p$. The lemma uses the notion of a strongly good $m$-cell, which was defined to be an $m$-cell $D$ such that $I_{Bm/p}(D)$ holds for a large constant $B$.

\begin{lemma}\label{le:mu}
There exists $\delta>0$ and a function $\mu:(0,\delta)\to(0,\infty)$ satisfying $\mu(p)=\lambda+o(1)$, such that if $\hat{K}(p)=\exp\big(\mu(p)/p\big)$ then
\begin{enumerate}
\item the probability that a $\hat{K}(p)$-cell contains an internally spanned critical droplet is $1/2+o(1)$ as $p\to 0$, and
\item the probability that a $\hat{K}(p)$-cell is strongly good is at least $1/2$ for all $p\in(0,\delta)$, provided $B\geq 6$.
\end{enumerate}
\end{lemma}

\begin{proof}
Let $\epsilon>0$. We shall prove the existence of a function $\mu(p)$ such that the probability that a $\hat{K}(p)$-cell contains an internally spanned critical droplet is at least $1/2+\epsilon$ and at most $1/2+2\epsilon$. This will prove (i), and the step up from $\Gamma\big(\hat{K}(p)\big)$ to $\strongly\big(\hat{K}(p)\big)$ required for (ii) will be provided by Lemma \ref{le:holroydplus}.

We use Theorem $1$ of \cite{Hol}, in the following form: if $\liminf p\log n > \lambda$ then $\P_p\big(I(n)\big)\to 1$, while if $\limsup p\log n < \lambda$ then $\P_p\big(I(n)\big)\to 0$. If $[n]^2$ is internally spanned then certainly it contains an internally spanned critical droplet, by the rectangles process, Lemma \ref{le:AL}, so it follows that if $\liminf p\log n > \lambda$ then $\P_p\big(\Gamma(n)\big)\to 1$. For a corresponding statement from below we need to use the proof of Theorem $1$ in \cite{Hol}, rather than the statement of the theorem itself. The proof shows that if $\limsup p\log n \leq (1-\epsilon)\lambda$, then with high probability $[n]^2$ does not contain an internally spanned droplet with long side between $C/p$ and $2C/p$, where $C$ is a large constant depending on $\epsilon$. Thus, by the rectangles process, under the same assumptions, with high probability $[n]^2$ does not contain an internally spanned critical droplet.

It follows that, for any $\epsilon>0$ and any $\epsilon'>0$, if $p$ is sufficiently small, then
\begin{equation}\label{eq:inv1}
\P_p\bigg(\Gamma\Big( e^{\frac{\lambda+\epsilon'}{p}} \Big)\bigg) > 1-\epsilon
\end{equation}
and
\begin{equation}\label{eq:inv2}
\P_p\bigg(\Gamma\Big( e^{\frac{\lambda-\epsilon'}{p}} \Big)\bigg) < \epsilon.
\end{equation}

Now, cover $[n+1]^2$ with one copy of $[n]^2$, two copies of $[n]\times[\gamma]$, and one copy of $[\gamma]\times[n]$. Observe that if $[n+1]^2$ contains a critical droplet, then so must at least one of the four covering sets, so we may deduce that
\[
\P_p\big(\Gamma(n+1)\big) \leq \P_p\big(\Gamma(n)\big) + 3\P_p\big(\Gamma([n]\times[\gamma])\big).
\]
By tiling $[n]^2$ with $[n]\times[\gamma]$ rectangles we have
\[
1-\P_p\big(\Gamma(n)\big) \leq \Big(1-\P_p\big(\Gamma([n]\times[\gamma])\big)\Big)^{n/\gamma},
\]
so $\P_p\big(\Gamma(n)\big) \gg \P_p\big(\Gamma([n]\times[\gamma])\big)$ provided $n\gg\gamma$. Hence, if $p$ is sufficiently small, then
\[
\P_p\big(\Gamma(n+1)\big) \leq (1+\epsilon)\P_p\big(\Gamma(n)\big).
\]
Together with \eqref{eq:inv1} and \eqref{eq:inv2} it follows that there exists a function $\mu(p)=\lambda+o(1)$ such that the probability that a $\hat{K}(p)$-cell contains an internally spanned critical droplet is at least $1/2+\epsilon$ and at most $1/2+2\epsilon$. As previously observed, this proves (i), and (ii) now follows from Lemma \ref{le:holroydplus}.
\end{proof}

We are now in a position to define the critical grid size $K$.

\begin{definition}\label{de:K}
Let $p_0>0$ be a quantity to be determined later, but which certainly satisfies $p_0<\delta$, where $\delta$ is as in Lemma \ref{le:mu}. The critical grid size $K=K(p)$ is defined by
\[
K(p) = \begin{cases} \exp\big(\mu(p)/p\big) & \text{if } p\leq p_0 \\ \exp\big(\mu(p_0)/p_0\big) & \text{if } p>p_0. \end{cases}
\]
(Thus, $K(p)$ is equal to the function $\hat{K}(p)$ defined in Lemma \ref{le:mu} when $p\leq p_0$.)
\end{definition}

The purpose of the parameter $p_0$ is that by taking $p_0$ sufficiently small we obtain an arbitrarily large lower bound for $K(p)$ uniformly over all $p\in(0,1)$. Furthermore, by Lemma \ref{le:mu} (for small $p$) and by coupling (for $p\geq p_0$), a $K(p)$-cell is strongly good with probablity at least $1/2$ for all $p$.

The function $\mu$ whose existence we have just proved in Lemma \ref{le:mu} is the precisely the function $\mu$ whose existence is asserted in Theorem \ref{th:small}. Thus, the significance of the factor of $K=\exp\big(\mu(p)/p\big)$ in the formula for $T$ in that theorem is that it is a grid size at which the probability of percolation is a constant.

The two functions $\mu$ and $K$ will continue to be used extensively throughout the paper, so it is worth bearing in mind their key properties: these are that $\mu(p)$ is equal to $\lambda+o(1)$ and that the probability a $p$-random subset of $[K]^2$ percolates is approximately constant.

\section{Large $p$}\label{se:large}

The lower bound in Theorem \ref{th:large} is better described as an observation. We state it here as a separate lemma so that it can be reused for part of the proof of the lower bound in Theorem \ref{th:small}. Here, and throughout the paper, we use $q$ to denote $1-p$, the probability that a site is initially uninfected.

\begin{lemma}\label{le:lb}
Let $p=p(n)$ be probabilities such that $\log 1/q \ll \log n$, and let $T$ be the percolation time of a $p$-random subset of $[n]^2$. Then
\[
T \geq \frac{(1+o(1))\log n}{2\log 1/q}
\]
with high probability as $n\to\infty$.
\end{lemma}

\begin{proof}
Let $t\geq 1$. Divide $[n]^2$ into $n^2/(4t+2)$ disjoint $(2t+1)\times 2$ rectangles. If any one of these rectangles is initially empty, then the middle two squares in that rectangle cannot be infected before the $t$th step, so $T$ must be at least $t$. Hence
\[
\P_p(T \leq t) \leq (1-q^{(4t+2)})^{n^2/(4t+2)}.
\]
The right-hand side is at most
\[
\exp\big[-q^{(4t+2)} n^2/(4t+2)\big] = \exp\Big(-\exp\big(2\log n - \log(4t+2) - (4t+2)\log(1/q)\big)\Big),
\]
which is $o(1)$ if
\[
\limsup_{n\to\infty} \frac{2t\log 1/q}{\log n} < 1. \qedhere
\]
\end{proof}

Now we begin the build up to the proof of the upper bound in Theorem \ref{th:large}. In the previous section we established the existence of the critical grid size $K=\exp\big(\mu(p)/p\big)$, where by ``critical'' in this context we mean that $p$-random initial subsets of cells of this side length percolate with probability bounded away from $0$ and $1$. Here we use this critical grid size as the base case of an induction argument which proves the existence of a larger, but not considerably larger, grid size $L$, with the property that the probability cells of this side length fail to percolate is essentially equal to the probability of the existence of an empty double row or column of the same length. The reason we want $L$ not to be too large is because it appears as an error term in the proof of the upper bound in Theorem~\ref{th:large}.

The following inequality will form the basis of the induction argument we use to prove the important properties of $L$. Recall that we write $\eta_m$ for the probability that an $m$-cell $D$ is bad, where bad was defined to mean that $\interior(D)\not\subset[D\cap A]_{Bm/p}$.

\begin{lemma}\label{le:etaineq}
If $B\geq 50$, then for all $m\geq 1$ we have
\begin{equation}\label{eq:eta}
\eta_{2m} \leq \eta_m^4 + 100 m^2 q^{4m-8}.
\end{equation}
\end{lemma}

\begin{proof}
Suppose a $2m$-cell $D=[(1,1),(2m,2m)]$ is bad, and divide it into four disjoint $m$-cells $D_1$, $D_2$, $D_3$ and $D_4$, with bottom-left corners at $(1,1)$, $(1,m+1)$, $(m+1,1)$ and $(m+1,m+1)$ respectively. Either all four $m$-cells are bad, or at least one of them is good. Suppose one of them is good, say $D_1$.

Let $S$ be a droplet such that $\lg(S)=m+1$ and $\sh(S)=m-2$. Suppose first that $S$ is taller than it is wide, so that $\dim(S)=(m-2,m+1)$. In that case we say that $S$ is \emph{traversable} if it has no empty double rows; so if, without loss of generality, $S=[(1,1),(m-2,m+1)]$, then $S$ is traversable if the sets $[(1,1),(m-2,2)]$, $[(1,2),(m-2,3)]$, $\dots$, $[(1,m),(m-2,m+1)]$ are all occupied. We say that $S$ is \emph{quickly traversable} if every site in $S$ except those in its topmost row is infected by time $24m/p$ assuming the column immediately below $S$ is initially full. So again, if $S=[(1,1),(m-2,m+1)]$, $S'=[(1,0),(m-2,0)]$ and $S''=[(1,m+1),(m-2,m+1)]$, then $S$ is quickly traversable if $S'\cap(S\setminus S'')\subset[S'\cup(A\cap S)]_{24m/p}$. If instead $S$ is oriented so that $\dim(S)=(m+1,m-2)$ then similarly we say that $S$ is \emph{traversable} if it has no empty double columns and \emph{quickly traversable} if every site in $S$ except those in its leftmost column is infected by time $24m/p$ assuming the column immediately to the left of it is initially full.

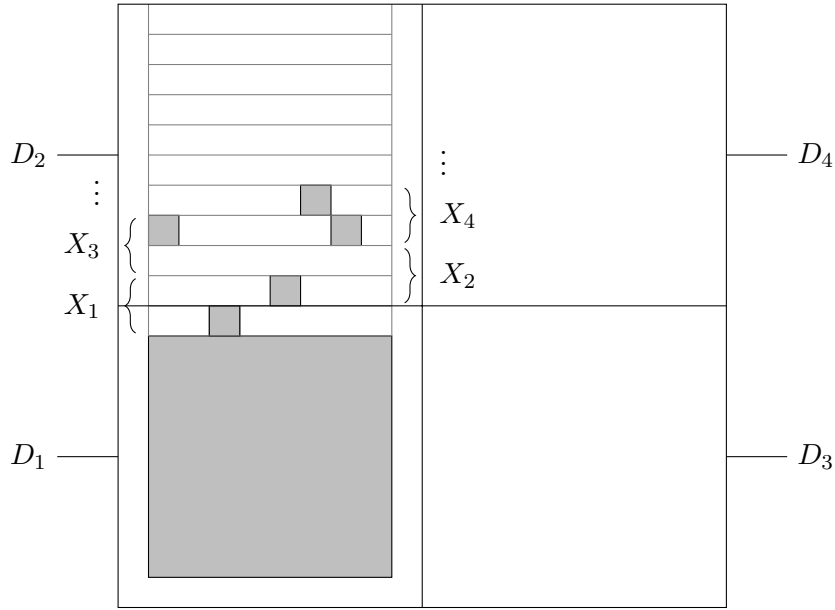
\begin{figure}[ht]
  \centering
  \begin{tikzpicture}[scale=0.8,>=latex]
    \draw [fill=gray!50] (0.5,0.5) rectangle (4.5,4.5);
    \draw [fill=gray!50] (1.5,4.5) rectangle (2,5);
    \draw [fill=gray!50] (2.5,5) rectangle (3,5.5);
    \draw [fill=gray!50] (0.5,6) rectangle (1,6.5);
    \draw [fill=gray!50] (3,6.5) rectangle (3.5,7);
    \draw [fill=gray!50] (1.5,4.5) rectangle (2,5);
    \draw [fill=gray!50] (3.5,6) rectangle (4,6.5);
    \draw [gray] (0.5,4.5) -- (0.5,10) (4.5,4.5) -- (4.5,10);
    \foreach \x in {4.5,5,...,9.5}
      \draw [gray] (0.5,\x) -- (4.5,\x);
    \draw [decorate,decoration={brace,amplitude=4pt},xshift=-6pt] (0.5,4.55) -- (0.5,5.45) node [black,midway,xshift=-20pt] {$X_1$};
    \draw [decorate,decoration={brace,amplitude=4pt},xshift=-6pt] (0.5,5.55) -- (0.5,6.45) node [black,midway,xshift=-20pt] {$X_3$};
    \draw [decorate,decoration={brace,amplitude=4pt,mirror},xshift=6pt] (4.5,5.05) -- (4.5,5.95) node [black,midway,xshift=20pt] {$X_2$};
    \draw [decorate,decoration={brace,amplitude=4pt,mirror},xshift=6pt] (4.5,6.05) -- (4.5,6.95) node [black,midway,xshift=20pt] {$X_4$};
    \node at (-0.35,7) {$\vdots$};
    \node at (5.35,7.5) {$\vdots$};
    \draw (0,0) rectangle (10,10);
    \draw (5,0) -- (5,10) (0,5) -- (10,5);
    \draw (0,2.5) -- (-1,2.5) node [left] {$D_1$};
    \draw (0,7.5) -- (-1,7.5) node [left] {$D_2$};
    \draw (10,2.5) -- (11,2.5) node [right] {$D_3$};
    \draw (10,7.5) -- (11,7.5) node [right] {$D_4$};
  \end{tikzpicture}
  \caption{The $2m$-cell $D$ is divided into four $m$-cells; $D_1$, which is semi-good (its interior is shown shaded), together with $D_2$, $D_3$ and $D_4$. The aim is to grow upwards from $D_1$ into $D_2$ quickly, by considering the left-most initially infected site in each double row above $\interior(D_1)$ in turn. In this figure, $X_1=3$, $X_2=5$ and $X_3=X_4=1$.}\label{fi:etaineq}
\end{figure}

The six droplets to which we need to apply these definitions are the following:
\begin{align*}
S_1 &= [(2,m),(m-1,2m)], & S_2 &= [(1,m+2),(m+1,2m-1)], \\
S_3 &= [(m+2,1),(2m-1,m+1)], & S_4 &= [(m,2),(2m,m-1)], \\
S_5 &= [(m+2,m),(2m-1,2m)], & S_6 &= [(m,m+2),(2m,2m-1)].
\end{align*}
Clearly droplets $S_1$, $S_3$ and $S_5$ all have dimensions $(m-2,m+1)$, while droplets $S_2$, $S_4$ and $S_6$ all have dimensions $(m+1,m-2)$. (See Figure \ref{fi:twoSs}.)

\begin{figure}[ht]
  \centering
  \begin{tikzpicture}[scale=0.8,>=latex]
    \draw [fill=gray!50] (0.5,0.5) rectangle (4.5,4.5);
    \draw [very thick] (0.5,4.5) rectangle (4.5,10);
    \draw [very thick] (4.5,5.5) rectangle (10,9.5);
    \foreach \x in {5.5,6,...,9.5}
      \draw [gray] (0.5,\x) -- (4.5,\x) (\x,5.5) -- (\x,9.5);
    \draw (0,0) rectangle (10,10);
    \draw (5,0) -- (5,10) (0,5) -- (10,5);
    \draw (0,2.5) -- (-1,2.5) node [left] {$D_1$};
    \draw (0,7.5) -- (-1,7.5) node [left] {$D_2$};
    \draw (10,2.5) -- (11,2.5) node [right] {$D_3$};
    \draw (10,7.5) -- (11,7.5) node [right] {$D_4$};
    \draw (0.5,5.75) -- (-1,5.75) node [left] {$S_1$};
    \draw (8.75,5.5) -- (8.75,4.25) -- (11,4.25) node [right] {$S_6$};
  \end{tikzpicture}
  \caption{Suppose that $D_1$ is good, that among every disjoint pair of droplets $S_i$ and $S_j$ at least one is traversable, and that no droplet $S_i$ is traversable but not quickly traversable. Then $D$ is also good.}\label{fi:twoSs}
\end{figure}
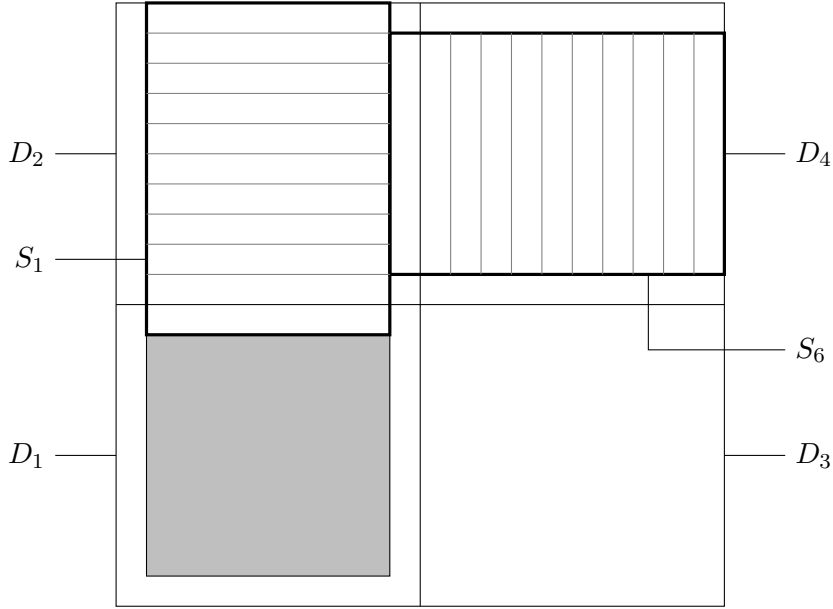

It is easy to see that if $D_1$ is good but $D$ is bad then either at least two disjoint $S_i$ are not traversable, or at least one of the $S_i$ is traversable but not quickly traversable. Indeed, were at least two of the $S_i$ quickly traversable, then the total time it would take $D$ to fill would be at most $Bm/p + 48m/p + m < 2Bm/p$, since $B\geq 50$, so $D$ would be good.

The probability that at least two disjoint $S_i$ are not traversable is at most $15m^2(1-p)^{4m-8}$. To bound the probability that at least one of the $S_i$ is traversable but not quickly traversable, suppose without loss of generality that $S_1$ has this property, and let $X_i$ be the position of the first initially infected site (or pair of sites) along the $i$th double row $[(2,m+i-1),(m-1,m+i)]$, counting from the left, for $i=1,\dots,m$ (see Figure \ref{fi:etaineq}). Note that the time it takes $S_1$ to fill given that the row immediately below it is initially full is at most
\[
2\sum_{i=1}^m X_i + m.
\]
If this quantity is greater then $24m/p$ then, crudely,
\[
\max \left\{ \sum_{i \text{ odd}} X_i , \sum_{i \text{ even}} X_i \right\} > \frac{10m}{p}.
\]
Each of the sums on the left-hand side consists of independent Geometric random variables, so the probability that one of the sums is greater than $10m/p$ is just the probability that a $\bin(10m/p,p)$ random variable is less than $m/2$. This probability is
\begin{align}\label{eq:binombound}
\P\left(\bin\left(\frac{10m}{p},p\right)<\frac{m}{2}\right) &= \sum_{k=0}^{m/2} \binom{10m/p}{k} p^k (1-p)^{10m/p-k} \notag \\
&\leq 2 \binom{10m/p}{m/2} p^{m/2} (1-p)^{10m/p-m/2} \notag \\
&\leq 2 (20e)^{m/2} (1-p)^{10m/p-m/2} \notag \\
&\leq \exp\Big( -\big( (9/p)\log(1/q) - (1/2)\log(40e) \big) m \Big).
\end{align}
The inequality
\[
(9/p)\log 1/q - (1/2)\log(40e) > 4\log 1/q
\]
holds for all $p\in(0,1)$, so \eqref{eq:binombound} is at most $\exp(-4m\log 1/q)$ for all $p\in(0,1)$. Thus, the probability that at least one of the $S_i$ is traversable but not quickly traversable is at most $6\exp(-4m\log 1/q)$.

Putting these observations together, and noting that there were four choices for the good square, we have
\[
\eta_{2m} \leq \eta_m^4 + 60 m^2 q^{4m-8} + 24 q^{4m},
\]
which completes the proof.
\end{proof}

The inequality we have just derived in Lemma \ref{le:etaineq} is the tool that will drive the induction argument in the next lemma to prove the key property of the grid size $L$, which is that the probability percolation fails correlates with the probability of the existence of an empty double row or column. Before stating the lemma we define the grid size $L$.

\begin{definition}\label{de:L}
We define
\[
L:=6K^2\log 1/q.
\]
\end{definition}

\begin{lemma}\label{le:etaineq2}
The probability $\eta_L$ that an $L$-cell is bad satisfies the inequality
\[
\eta_L\leq 50L^2q^{-8}\exp(-2L\log 1/q).
\]
\end{lemma}

\begin{proof}
By Lemma \ref{le:etaineq}, for any $m\geq 1$,
\[
\eta_{2m} \leq \max\big\{ 2\eta_m^4 \, , \, 200m^2q^{4m-8} \big\}.
\]
Since $\eta_K\leq 1/2$ by the definition of $K$, it follows by induction that
\[
\eta_{2^r K} \leq \max\big\{ 2^{-4^r + (4^r-1)/3} \, , \, 200 (2^{r-1}K)^2q^{4\cdot 2^{r-1}K-8} \big\}.
\]
Somewhat crudely, $\eta_{2^r K}$ is at most the second term in the maximum on the right-hand side if $(2/3) 4^r \log 2 \geq 2^{r+1}K \log 1/q$, which holds if 
\[
2^r K \geq \frac{3}{\log 2} K^2 \log 1/q.
\]
Since $L> (3/\log 2) K^2 \log 1/q$, we conclude that
\[
\eta_L \leq 50 L^2 q^{-8} \exp\big( -2L\log 1/q \big),
\]
as required.
\end{proof}

The next ingredient we need is a technical lemma which will allow us to prove the convergence of a certain geometric series. This is the only point in the proof where we use the fact that $L$ is not too small.

\begin{lemma}\label{le:ratio}
Let $C>0$ be a constant, let $p_0$ (in the definition of $K$) be sufficiently small, and let $A$ (in the definition of $L$) be sufficiently large. Then
\[
\big( CL^2q^{-8} \big)^{1/L} (1-p)^{1/8} < 1.
\]
\end{lemma}

\begin{proof}
We require $CL^2q^{-8} < q^{-L/8}$, or, since $L=AK^2\log 1/q$, equivalently we require
\begin{equation}\label{eq:ratio}
\log(CA^2) + 4\log K + 2\log\log 1/q + 8\log 1/q < \frac{A}{8} K^2 (\log 1/q)^2.
\end{equation}
By taking $p_0$ sufficiently small in Definition \ref{de:K}, we obtain $\log K < K^2$ for all $p$. Thus, the left-hand side of \eqref{eq:ratio} is at most
\[
\log(CA^2) + 4\log K + 10\log 1/q \leq 40 (\log CA^2) (\log K) (\log 1/q) < AK^2\log 1/q
\]
provided $A$ is sufficiently large.
\end{proof}

An \emph{up-right $m$-path} is a sequence of $m$-cells $D_1,\dots,D_u$ such that, for $1\leq i\leq u-1$, if $D_i=[(a,b),(c,d)]$ then $D_{i+1}$ is either equal to $[(a+m,b),(c+m,d)]$ or to $[(a,b+m),(c,d+m)]$. Thus, the bottom-left corner of $D_{i+1}$ is obtained from the bottom-left corner of $D_i$ by adding $m$ to exactly one of its coordinates, so the $m$-cells are disjoint, but consecutive cells are touching. The \emph{length} of the up-right $m$-path $D_1,\dots,D_u$ is $u$. An \emph{up-right path} is simply an up-right $1$-path, and we do not distinguish $1$-cells from sites.

The next lemma is the key step in the proof of the lower bound of the large $p$ theorem, and one of the most important lemmas in the paper. The bootstrap process is restricted to the positive quadrant of the plane and we ask how likely it is that the origin is uninfected at time $t$. We show that the answer is that it is roughly the same as the probability of there being an empty single row or column of length about $t$ starting at the origin. This latter event clearly implies that the origin is uninfected at time (about) $t$, so the interest is that the contribution to the probability from other configurations which also guarantee the origin is uninfected at time $t$ is small.

The idea behind the proof is as follows. If the origin is uninfected at time $t$ then there must exist an up-right path of length $t$ of initially uninfected sites starting at the origin. Unfortunately the crude way of estimating the probability of this event --- by taking a union bound over all paths --- gives much too large an estimate; we get $2^t(1-p)^t$, and the problem here lies in the combinatorial factor of $2^t$. To overcome this, we tile the positive quadrant with $L$-cells and wait an initial time $t'=BL/p$ so that all good $L$-cells will have filled (possibly except for their edges). We then look at the original up-right path of initially uninfected sites, and observe that the first $t-t'$ sites in that path (counting from the origin) must be uninfected at time $t'$. Now consider the up-right $L$-path of length $(t-t')/L$ induced by our up-right path of length $t-t'$. Each $L$-cell either is bad or intersects the up-right path only on its edges. In the first case we gain a probability of $(1-p)^{2L}$ (up to a polynomial correction) from Lemma \ref{le:etaineq2}. There is still a combinatorial factor involved in choosing the $L$-cells, but it is smaller than before, and is beaten by the gain in probability of a factor of $(1-p)^{L}$ over what we would have obtained had the path of uninfected sites stayed on the edge of the quadrant (this is where we use Lemma \ref{le:ratio}). In the second case, if $l$ is the total length of the path along edges of $L$-cells (which need not be consecutive) then we obtain a probability close to $(1-p)^{2l}$, because the up-right path is restricted to long, straight segments, and these must be part of double empty rows or columns. Furthermore, the highly restricted nature of the path also implies that there is only a small combinatorial loss. In both cases, the probability is much smaller than it would have been had the up-right path remained on one of the edges of the quadrant.

\begin{lemma}\label{le:upright}
Let $p\in(0,1)$ and $t\in\N$, and define $t'=BL(p)/p$, $D=[(0,0),(t,t)]$ and let $A\sim\bin(D,p)$. Then the probability that $(0,0)$ is uninfected at time $t$ is at most
\[
\frac{16(1-p)^{t-t'}}{p}.
\]
\end{lemma}

\begin{proof}
Suppose the origin is uninfected at time $t$. If a site $y$ is such that both $y+e_1$ and $y+e_2$ are infected at some time $s$, then $y$ is certainly infected at time $s+1$. It follows that there exists an up-right path $x_1,\dots,x_{t-t'+1}$ of uninfected sites at time $t'$ with $x_1=(0,0)$. Let $k$ be maximal such that both coordinates of $x_{t-t'+1-k}$ are non-zero, or if there is no such $k$ then set $k=0$. Thus $k=0$ corresponds to the existence of an unoccupied straight line of length $t-t'+1$ with one endpoint at the origin. We shall show that the event $k=0$ is the most likely way of ensuring that the origin is uninfected at time $t$.

The up-right path $x_{t-t'-k},\dots,x_{t-t'+1}$ intersects an up-right $L$-path $D_1,\dots,D_\tau$, where $x_{t-t'-k}$ is the bottom-left site of $D_1$ and $\tau\geq k/L$. Since $x_1,\dots,x_{t-t'+1}$ are uninfected at time $t'=BL/p$, none of the $L$-cells $D_1,\dots,D_\tau$ is strongly good, so each is either semi-good or bad.

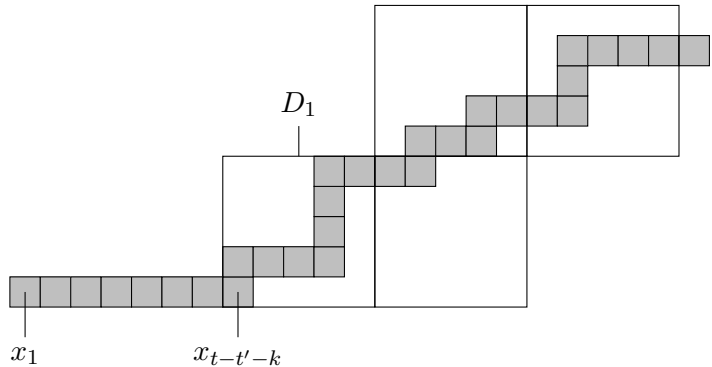
\begin{figure}[ht]
  \centering
  \begin{tikzpicture}[scale=0.4,>=latex]
    \tikzstyle{site}=[draw,fill=gray!50];
    \draw [site] (0,0) rectangle (1,1);
    \draw [site] (1,0) rectangle (2,1);
    \draw [site] (2,0) rectangle (3,1);
    \draw [site] (3,0) rectangle (4,1);
    \draw [site] (4,0) rectangle (5,1);
    \draw [site] (5,0) rectangle (6,1);
    \draw [site] (6,0) rectangle (7,1);
    \draw [site] (7,0) rectangle (8,1);
    \draw [site] (7,1) rectangle (8,2);
    \draw [site] (8,1) rectangle (9,2);
    \draw [site] (9,1) rectangle (10,2);
    \draw [site] (10,1) rectangle (11,2);
    \draw [site] (10,2) rectangle (11,3);
    \draw [site] (10,3) rectangle (11,4);
    \draw [site] (10,4) rectangle (11,5);
    \draw [site] (11,4) rectangle (12,5);
    \draw [site] (12,4) rectangle (13,5);
    \draw [site] (13,4) rectangle (14,5);
    \draw [site] (13,5) rectangle (14,6);
    \draw [site] (14,5) rectangle (15,6);
    \draw [site] (15,5) rectangle (16,6);
    \draw [site] (15,6) rectangle (16,7);
    \draw [site] (16,6) rectangle (17,7);
    \draw [site] (17,6) rectangle (18,7);
    \draw [site] (18,6) rectangle (19,7);
    \draw [site] (18,7) rectangle (19,8);
    \draw [site] (18,8) rectangle (19,9);
    \draw [site] (19,8) rectangle (20,9);
    \draw [site] (20,8) rectangle (21,9);
    \draw [site] (21,8) rectangle (22,9);
    \draw [site] (22,8) rectangle (23,9);
    \draw (7,0) rectangle (12,5);
    \draw (12,0) rectangle (17,5);
    \draw (12,5) rectangle (17,10);
    \draw (17,5) rectangle (22,10);
    \draw (0.5,0.5) -- (0.5,-1) node [below] {$x_1$};
    \draw (7.5,0.5) -- (7.5,-1) node [below] {$x_{t-t'-k}$};
    \draw (9.5,5) -- (9.5,6) node [above] {$D_1$};
  \end{tikzpicture}
  \caption{The shaded squares are sites forming an up-right path $(0,0)=x_1,x_2,\dots,x_{t-t'+1}$. In this example, $L=5$ and the four $5$-cells shown are the up-right $L$-path $D_1$, $D_2$, $D_3$, $D_4$.}\label{fi:upright}
\end{figure}

Let $E_2(i,j)$ denote the event that the $L$-cells $D_i,\dots,D_j$ are semi-good, and that at time $t'$ there exists an up-right path $y_1\dots,y_u$ of uninfected sites entirely contained within $\partial D_i\cup\dots\cup\partial D_j$, with $y_1$ an element of either the bottom or left edge of $D_i$ and $y_u$ an element of either the top or right edge of $D_j$.

Given an $m$-cell $D=[(a,b),(c,d)]$, let the \emph{left buffer} of $D$ be the $2\times(m-2)$ rectangle $[(a-1,b+1),(a,d-1)]$, and define similarly the \emph{right}, \emph{top} and \emph{bottom buffers} of $D$. Observe that, given adjacent $m$-cells $S_1$ and $S_2$ in an up-right $m$-path, either the right buffer of $S_1$ is the same as the left buffer of $S_2$, or the top buffer of $S_1$ is the same as the bottom buffer of $S_2$. Now suppose that among $D_1,\dots,D_\tau$ there is a sequence of $r$ consecutive semi-good cells $D_i,\dots,D_{i+r-1}$, so that the event $E_2(i,i+r-1)$ occurs. Let $\mathcal{B}$ denote the set of buffers of $D_i,\dots,D_{i+r-1}$, excluding the left and bottom buffers of $D_i$ and the top and right buffers of $D_{i+r-1}$. Since the interiors of $D_i,\dots,D_{i+r-1}$ are all full by time $t'$, the existence of an up-right path along the edges of $D_i,\dots,D_{i+r-1}$ of sites uninfected at time $t'$ implies that at least $r-1$ of the buffers in $\mathcal{B}$ were initially unoccupied. The reason for this is that if one considers sides of an $L$-cell to have unit length, then the $\ell_1$ distance between either the top-left or the bottom-right corner of $D_i$ and either the top-left or the bottom-right corner of $D_{i+r-1}$ is equal to $r-1$. Crucially, by the definition of $k$, these unoccupied buffers are all subsets of $D$. Each buffer is a set of $2(L-2)$ sites, so
\begin{equation}\label{eq:E2}
\P_p\big(E_2(i,i+r-1)\big) \leq 2^{r-1} (1-p)^{2(r-1)(L-2)}.
\end{equation}
(Had we chosen $k$ differently, some of the buffers may have been only half contained in $D$, which would render this bound incorrect. In other words, this is the point in the argument where, rather subtly, we use the fact that the up-right path has moved away from the boundary of $D$.) The bound in \eqref{eq:E2} does not give any information when $r=1$, but in that case we still have $\P_p\big(E_2(i,i)\big)\leq 4(1-p)^L$, since $D_i$ is only semi-good, so at least one of its edges is empty.

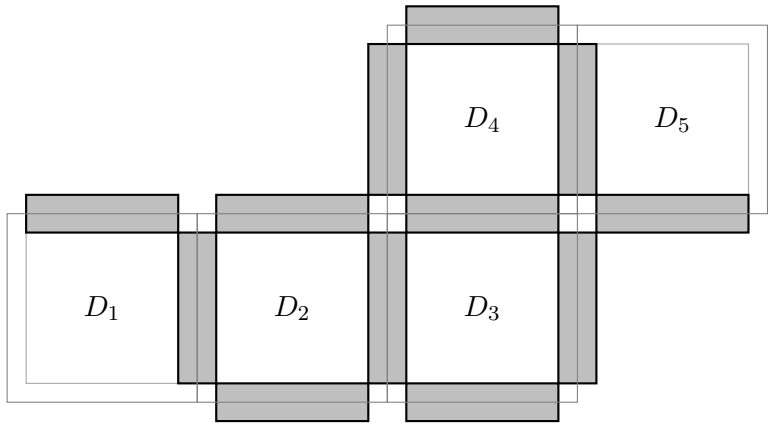
\begin{figure}[ht]
  \centering
  \begin{tikzpicture}[scale=0.5,>=latex]
    \tikzstyle{buffer}=[thick,fill=gray!50];
    \tikzstyle{cell}=[gray];
    \tikzstyle{int}=[gray!70];
    \draw [int] (0.5,0.5) rectangle (4.5,4.5);
    \draw [int] (5.5,0.5) rectangle (9.5,4.5);
    \draw [int] (10.5,0.5) rectangle (14.5,4.5);
    \draw [int] (10.5,5.5) rectangle (14.5,9.5);
    \draw [int] (15.5,5.5) rectangle (19.5,9.5);
    \draw [buffer] (5.5,-0.5) rectangle (9.5,0.5);
    \draw [buffer] (10.5,-0.5) rectangle (14.5,0.5);
    \draw [buffer] (4.5,0.5) rectangle (5.5,4.5);
    \draw [buffer] (9.5,0.5) rectangle (10.5,4.5);
    \draw [buffer] (14.5,0.5) rectangle (15.5,4.5);
    \draw [buffer] (0.5,4.5) rectangle (4.5,5.5);
    \draw [buffer] (5.5,4.5) rectangle (9.5,5.5);
    \draw [buffer] (10.5,4.5) rectangle (14.5,5.5);
    \draw [buffer] (15.5,4.5) rectangle (19.5,5.5);
    \draw [buffer] (9.5,5.5) rectangle (10.5,9.5);
    \draw [buffer] (14.5,5.5) rectangle (15.5,9.5);
    \draw [buffer] (10.5,9.5) rectangle (14.5,10.5);
    \draw [cell] (0,0) rectangle (5,5);
    \draw [cell] (5,0) rectangle (10,5);
    \draw [cell] (10,0) rectangle (15,5);
    \draw [cell] (10,5) rectangle (15,10);
    \draw [cell] (15,5) rectangle (20,10);
    \node at (2.5,2.5) {$D_1$};
    \node at (7.5,2.5) {$D_2$};
    \node at (12.5,2.5) {$D_3$};
    \node at (12.5,7.5) {$D_4$};
    \node at (17.5,7.5) {$D_5$};
  \end{tikzpicture}
  \caption{An up-right $L$-path of length five. Buffers in the set $\mathcal{B}$ for this path are the shaded rectangles.}\label{fi:semi}
\end{figure}

Let $E_1(i,j)$ denote the event that all of the $L$-cells $D_i,\dots,D_j$ are bad. By Lemma \ref{le:etaineq2}, the probability of $E_1(i,i+r-1)$ is at most $f(p)^r (1-p)^{2Lr}$, where $f(p)=50L^2q^{-8}$.

Now, there exists a finite sequence $0=b_1<s_1<b_2<s_2<b_3<\dots$, where the last term is equal to $\tau$, such that the event
\[
E = E_1(b_1,s_1-1) \cap E_2(s_1,b_2-1) \cap E_1(b_2,s_2-1) \cap E_2(s_2,b_3-1) \cap \dots
\]
occurs. Suppose that the last term of the sequence is $\tau=b_{u+1}-1$; the argument is similar if $\tau=s_{u+1}-1$. Let $v$ be the number of $i$ for which $b_{i+1}-1=s_i$; thus, $v$ is the number of times that there are three consecutive $L$-cells in the sequence $D_1,\dots,D_\tau$ that are of the form bad, semi-good, bad, in that order. We have
\begin{align*}
\P_p(E) &= \prod_{i=1}^u \P\big( E_1(b_i,s_i-1) \cap E_2(s_i,b_{i+1}-1) \big) \\
&\leq 4^v (1-p)^{Lv} \prod_{i=1}^u f(p)^{s_i-b_i} (1-p)^{2L(s_i-b_i)} 2^{b_{i+1}-s_1-1} (1-p)^{2(L-2)(b_{i+1}-s_i-1)} \\
&\leq \big(8f(p)\big)^\tau (1-p)^{(L-2)(2\tau - 2u + v)}.
\end{align*}
By partitioning sequences of consecutive semi-good $L$-cells into those of length $1$ and those of length greater than $1$, and since $2v+3(u-v)\leq\tau$, it follows that $2u-v\leq 2\tau/3$. Thus,
\[
\P_p(E) \leq \big(8f(p)\big)^\tau (1-p)^{4(L-2)\tau/3}.
\]

For a given $k$, and hence a given $\tau$, there are $2^\tau$ choices of up-right path of $L$-cells, and a further $2^\tau$ ways of choosing whether each $L$-cell is semi-good or bad. Therefore the probability that there exists an up-right path of uninfected sites of length $k$ starting from a given site is at most
\[
\big(32f(p)\big)^{k/L} (1-p)^{4k(1-2/L)/3} \leq \big(32f(p)\big)^{k/L} (1-p)^{5k/4}.
\]
Hence, the probability of the event $F$, which we define to be that the origin is uninfected at time $t$ in bootstrap percolation on the square $D$, is at most
\[
2 \sum_{k=0}^{t-t'} (1-p)^{t-t'-k} \big(32f(p)\big)^{k/L} (1-p)^{5k/4}.
\]
By taking $p_0$ (in the definition of $K$) sufficiently small and $A$ (in the definition of $L$) sufficiently large, and applying Lemma \ref{le:ratio}, the common ratio in the geometric series above, which is $\big(32f(p)\big)^{1/L}(1-p)^{1/4}$, has value at most $(1-p)^{1/8}$. Therefore,
\[
\P_p(F) \leq \frac{2 (1-p)^{t-t'}}{1-(1-p)^{1/8}} \leq \frac{16 (1-p)^{t-t'}}{p}. \qedhere
\]
\end{proof}

It is now just a small step to proving the upper bound in Theorem \ref{th:large}. We apply Lemma \ref{le:upright} to each of the four sites in a $2\times 2$ square of sites uninfected at time $t-2$ (which is possible if there is a site not too close to the boundary which is uninfected at time $t$). The lemma implies that the probability these four sites stay uninfected that long is approximately the same as the probability that they are initially at the centre of a double empty row or column of length about $2t$, which is what we require. There is a little more work to do to take into account the sites near the boundary of the grid (this would not be necessary if we were working on the torus rather than the grid). These sites have a greater probability of being uninfected at time $t$, but this is negated by the relatively small number of them.

\begin{proof}[Proof of Theorem \ref{th:large}]
The lower bound of Theorem \ref{th:large} is Lemma \ref{le:lb}, so we only have to prove the upper bound.

Let $t'=BL/p+2$. Suppose a site $x$ is uninfected at time $t\geq 2$, and suppose first that $x$ is not within distance $t$ of the boundary of $[n]^2$. It is easy to check that $x$ must be contained in a $2\times 2$ square of sites uninfected at time $t-2$, say $x_1$, $x_2$, $x_3$ and $x_4$. Let $D_1$, $D_2$, $D_3$ and $D_4$ be the four $t$-cells such that $x_i\in D_i$ for each $i$ and $x_i\notin D_j$ if $j\neq i$. Since $x_i$ is uninfected at time $t-2$ in bootstrap percolation with initial set $A\sim\bin([n]^2,p)$, it is also uninfected at time $t-2$ when the initial set is restricted to $A\cap D_i$. Applying Lemma \ref{le:upright} to each $D_i$ in turn, we find the probability that $x_1$, $x_2$, $x_3$ and $x_4$ are all uninfected at time $t-2$ is at most
\[
\left(\frac{16 (1-p)^{t-t'}}{p}\right)^4.
\]

It follows that the probability there exists a site $x$, which is uninfected at time $t$, but which is not within distance $t$ of the boundary of $[n]^2$, is at most
\begin{equation}\label{eq:xmiddle}
4n^2 \left( \frac{16 (1-p)^{t-t'}}{p} \right)^4 = 4 \exp \big( 2\log n + 4\log 1/p - 4(t-t')\log 1/q + O(1) \big).
\end{equation}
This is $o(1)$ if
\begin{equation}\label{eq:tt'}
t-t' \geq \frac{(1+\epsilon)\log n}{2\log 1/q}.
\end{equation}

When $x$ is close to the boundary of $[n]^2$, the calculation is similar. The probability that $x$ is uninfected at time $t$ is much larger, but to compensate for this there are fewer choices for $x$. Briefly, there at most $4nt$ sites within distance $t$ of one of the sides of $[n]^2$, but not within the same distance of one of the corners. Each such site which is uninfected at time $t$ has an adjacent site in an appropriate direction which is uninfected at time $t-1$. Applying Lemma \ref{le:upright} to this pair of sites and taking the union bound gives the probability that any of these sites is uninfected at time $t-1$ is at most
\begin{equation}\label{eq:xside}
8nt \left( \frac{16 (1-p)^{t-t'}}{p} \right)^2 = \exp \big( \log n - 2(t-t')\log 1/q + o(\log n) \big) = o(1),
\end{equation}
if \eqref{eq:tt'} holds. Similarly, there are at most $4t^2$ sites within distance $t$ of one of the corners of $[n]^2$, and by Lemma \ref{le:upright}, each has probability at most $16 (1-p)^{t-t'}/p$ of being uninfected at time $t$. Taking a union bound, the probability any of these is uninfected at time $t$ is at most
\begin{equation}\label{eq:xcorner}
4t^2\frac{16 (1-p)^{t-t'}}{p} = \exp\big( -(t-t')\log 1/q + o(\log n) \big) = o(1),
\end{equation}
if \eqref{eq:tt'} is satisfied.

Combining \eqref{eq:xmiddle}, \eqref{eq:xside} and \eqref{eq:xcorner}, recalling that $t'=BL/p+2=O\big(K^2(\log 1/q)/p\big)$, and using the notation of the statement of the theorem, we have
\[
T \leq \frac{(1+o(1))\log n}{2\log 1/q} + O\bigg(\frac{K^2\log 1/q}{p}\bigg)
\]
with high probability as $n$ tends to infinity. The deduction of the upper bound in Theorem \ref{th:large} from this statement is simply the assertion is that if $\liminf p\log\log n > 2\lambda$ then
\[
\frac{K^2\log 1/q}{p} \ll \frac{\log n}{\log 1/q},
\]
which is an easy computation. This completes the proof of Theorem \ref{th:large}.
\end{proof}

\section{Upper bound for small $p$}\label{se:uppersmall}

In this section we prove the upper bound of Theorem \ref{th:small}. Recall that we describe the range of $p$ for which Theorem \ref{th:small} applies (that is, $\liminf p\log n > \lambda$ and $p=o(1)$) as the ``small $p$ regime''.

Several of the lemmas we shall use in the proof of the upper bound in the small $p$ regime are similar to the lemmas used in the proof of the upper bound in the large $p$ regime. In fact, some (those which were covered in Section \ref{se:critgrids}) are identical, and for the rest (those which were covered in Section \ref{se:large}), we observe, omitting most of the details, that only small modifictions are required to adapt them to the small $p$ setting.

Our first lemma is an analogue to Lemma \ref{le:etaineq} in which ``bad'' is replaced by ``weakly bad'' (or equivalently, ``good'' is replaced by ``strongly good''). It is worth recalling that an $m$-cell was defined to be \emph{bad} if its interior is not contained in the span of the whole cell by time $Bm/p$, while it is \emph{weakly bad} if it is not internally spanned by the same time. Thus, the property of being bad is a stronger property of an $m$-cell than that of being weakly bad. (It is also worth recalling that we write $\theta_m$ for the probability that an $m$-cell is weakly bad.)

\begin{lemma}\label{le:thetaineq}
If $B\geq 50$, then for all $m\geq 1$ we have
\[
\theta_{2m} \leq \theta_m^4 + 50m^2 q^{2m}.
\]
\end{lemma}

\begin{proof}
The proof of this lemma is similar to the proof of Lemma \ref{le:etaineq}, except many of the details are simpler and so we only give a sketch. The advantage here is that we may assume that $D_1$ is strongly good, not just good, and this allows us to modify the meaning of \emph{traversable} so that it applies to the $m$-cells $D_1,\dots,D_4$, not to the $(m+1)\times(m-2)$ droplets $S_1,\dots,S_6$, and so that $D_i$ is traversable if all its \emph{single} rows and columns are occupied. Then the probability that at least two of the $m$-cells $D_2$, $D_3$, $D_4$ are not traversable is at most $12m^2(1-p)^{2m}$. The remainder of the proof is the same as that of Lemma \ref{le:etaineq}.
\end{proof}

The next definition and the lemma following it are the analogues of Definition \ref{de:L} (of the grid length $L$) and Lemma \ref{le:etaineq2}.

\begin{definition}\label{de:M}
Let $A$ be a large constant. For $n\geq K$, we define
\[
M := \max\left\{ A\sqrt{p\log(n/K)}K \, , \, A\log(n/K) \right\},
\]
where, as before, $K$ is the function defined in Definition \ref{de:K}.
\end{definition}

The definition of $M$, like the expression for $T$ in Theorem \ref{th:small}, is a maximum of two terms. (Of course, this is not a coincidence: Theorem \ref{th:small} says precisely that $T=\Theta(M/p)$.) As remarked in the introduction to the paper, the second of the two terms in the maximum (which there we called $t_2(n,p)$, so here it would be $pt_2(n,p)$) is only larger than the first, and therefore only relevant, when $\limsup p\log\log n\geq 2\lambda$. Thus, in the range in which Theorem \ref{th:large} does not supersede Theorem \ref{th:small}, the second term is only relevant when $p$ is approximately equal to $2\lambda/\log\log n$.
%, and therefore the central part of Theorem \ref{th:small} is the assertion that $T$ is concentrated to within a constant factor when the first of the two terms is the larger.

\begin{lemma}\label{le:thetaineq2}
There exist constants $c,C>0$ such that if $p$ is sufficiently small then the probability an $M$-cell is weakly bad satisfies
\[
\theta_M \leq \max\Big\{ \exp\big(-cp\log(n/K)\big) , \exp(-cpM) \Big\}.
\]
\end{lemma}

\begin{proof}
This time notice that the quantity $50K^2q^{2K}$ can be made arbitrarily small by taking $p$ sufficiently small. As in the proof of Lemma \ref{le:etaineq2}, we obtain
\[
\theta_{2^r K} \leq \max\big\{ 2^{-(2/3)4^r} , 100 (2^{r-1}K)^2q^{2\cdot 2^{r-1}K} \big\}.
\]
and hence
\begin{equation}\label{eq:thetamax}
\theta_M \leq \max\Big\{ \exp\big(-c(M/K)^2\big) , C(M/K)^2\exp\big(-c(M/K)Kp\big) \Big\}
\end{equation}
for constants $c,C>0$. Now $M\geq A\sqrt{p\log(n/K)}K$ by definition, so $(M/K)^2\geq A^2p\log(n/K)$. Also, $pM\gg\log M$. Hence
\[
\theta_M \leq \max\Big\{ \exp\big(-cp\log(n/K)\big) , \exp(-cpM) \Big\}
\]
with a different constant $c$.
\end{proof}

The final lemma we need before we can prove the upper bound in Theorem \ref{th:small}, and the only without an analogue in the large $p$ regime, is the following result which we shall use to bound the probability that there exists a large connected component of weakly bad $L$-cells. A proof can be found in \cite[pp.~129--132]{CiM}.

\begin{lemma}\label{le:coffeetime}
Let $G$ be a graph with maximum degree $d$. Then the number of connected induced subgraphs of $G$ of order $k$ that contain a given vertex is at most $(e(d-1))^k$.\qed
\end{lemma}

\begin{proof}[Proof of the upper bound in Theorem \ref{th:small}]
Tile $[n]^2$ with disjoint $M$-cells. After an initial time $t'=50M/p$, all uninfected sites will be contained in weakly bad $M$-cells. Consider the graph of $M$-cells in which there is an edge between two cells if they have a common side. Clearly this graph has maximum degree $4$. By Lemma \ref{le:coffeetime}, the probability there exists a connected component of weakly bad $M$-cells of order at least $k$ is at most
\begin{equation}\label{eq:probbadcomps}
\left(\frac{n}{M}\right)^2 (3e)^k \theta_M^k \leq \exp\big(2\log(n/M) + k\log\theta_M + O(k)\big).
\end{equation}
This quantity tends to zero if
\[
k \geq \frac{-3\log(n/M)}{\log\theta_M}.
\]

Recall from Lemma \ref{le:thetaineq2} that
\[
\theta_M \leq \max\Big\{ \exp\big(-cp\log(n/K)\big) , \exp(-cpM) \Big\}
\]
for constants $c,C>0$. Noting also that $\log(n/M)\geq\log(n/K)$, it follows that \eqref{eq:probbadcomps} is $o(1)$ provided $k$ satisfies
\[
k \geq \max\left\{ \frac{3}{cp},\frac{3\log(n/M)}{cpM} \right\}.
\]
So with $k$ equal to the maximum of these two expressions, with high probability the largest component of weakly bad $M$-cells has size at most $k$. Any component of $M$-cells has at least one cell with at least two sides not connected to the rest of the component, so given that all other cells are strongly good, that cell becomes infected after at most $2M$ additional time steps. Continuing, the entire component of weakly bad $M$-cells becomes fully infected in time at most $t=2Mk$. Hence, with high probability, the percolation time $T$ is at most
\[
t'+t \leq \frac{50M}{p} + \frac{6M}{cp} + \frac{6\log(n/M)}{cp} \leq C\sqrt{\frac{\log(n/K)}{p}}K + C\frac{\log n}{p}
\]
for some constant $C>0$.
\end{proof}

\section{Lower bound for small $p$}\label{se:lowersmall}

Recall that Theorem \ref{th:small} states that if $\liminf p\log n>\lambda$ and $p\to 0$ then $T=\Theta(M/p)$ with high probability, where $M$ was defined by
\[
M = \max\left\{ A\sqrt{p\log(n/K)}K \, , \, A\log(n/K) \right\}.
\]
In the previous section we proved the upper bound. Here we concentrate on the lower bound, and since we have already proved in Lemma \ref{le:lb} that $T=\Omega\big((\log n)/p\big)$ with high probability, we only have to prove that
\[
T \geq c\sqrt{\frac{\log(n/K)}{p}}K
\]
with high probability, for some constant $c>0$. Thus, in this section we shall always assume that $p$ is sufficiently small that $\sqrt{p\log(n/K)}K\geq\log(n/K)$, and hence that $M=A\sqrt{p\log(n/K)}K$.

At the basic level the idea behind our proof of the lower bound in the small $p$ regime is quite simple: we show that with high probability there exists a region of the grid in which the initial configuration $A$ is in some sense relatively sparse, and then that even if all the sites outside of this area are initially infected, the percolation time must still be quite large. A little more precisely, we shall find as large an area of the grid as possible not containing an internally spanned critical droplet (that is what we mean here by ``sparse''). Letting this area be $D$, we then generously take a new initial set $A'$ to consist of the closure of $D\cap A$ together with all sites outside of $D$, and observe that since $A\subset A'$, the percolation time of $A$ is certainly at least the percolation time of our new initial set $A'$.

How long should the set $A'$ take to percolate? The answer to that question depends on the shape of the droplet $D$, so the question we must answer first is: how should we choose the ratio of the sides of the droplet $D$ in order to maximize the expected percolation time of $A'$? There are two effects to balance. First, diagonal lines of infected sites are becoming infected from the corners of $D$ deterministically at rate $1$. Second, sites in $D$ are infected with density $p$, so we expect the sides of $D$ to become infected at rate $p$. After a moment's thought, one realizes this means that the optimal ratio of sides for $D$ should be $p:1$. The majority of this section of the paper deals with formalizing this heuristic: that the sides of $D$ should become infected at rate $p$.

Recall that $K=K(p)=\exp\big(\mu(p)/p\big)$, where $\mu(p)=\lambda+o(1)$. The function $M=A\sqrt{p\log(n/K)}K$ is designed so that the largest region $D$ of the grid that is likely not to contain an internally spanned critical droplet has area $M^2/p$. Combining this with our observation about the optimal ratio of the sides of this droplet, it follows that the droplet $D$ should have long side length $M/p$ and short side length $M$. We define an \emph{$M$-slab} to be any such droplet; thus, $D$ is an $M$-slab if $\lg(D)=M/p$ and $\sh(D)=M$.

Suppose that an $M$-slab is filled in time $t$ assuming that all sites outside the $M$-slab are initially infected, where $t$ is a large constant factor smaller than $M/p$. We ask what route the infection took from the edge of the $M$-slab to the centre. Suppose the route came via the bottom edge. Then we can say, deterministically, that there must be a sequence of internally spanned droplets such that together they do not leave an empty double row between the bottom edge of the $M$-slab and the centre, and such that the sum of the horizontal distances between consecutive internally spanned droplets is considerably smaller than one would expect. This says that part of our supposedly sparse $M$-slab is much more dense than even an average $M$-slab.

A sequence of droplets joining the boundary of an $M$-slab to the centre, such as the one described in the previous paragraph, is called a \emph{wave} (the definition is made precise in the next section). By counting the number of possible waves and estimating their probabilities, we show that the probability there exists a wave with small sum of horizontal distances between the droplets --- which is equivalent to the $M$-slab filling quickly --- is small. The details of the proof are long and technical, and the reader who is in a hurry may choose to omit them without losing the flow of the argument. (However, some of the definitions that occur alongside these arguments \emph{are} important, such as those of a wave and a slow $M$-slab.) The main part of the proof of the small $p$ lower bound theorem occurs in Section \ref{se:pflowersmall}.

\subsection{Waves and flood times of $M$-slabs}\label{se:waves}

The next definition is central to this part of the paper. It is the structure that we shall use to encode how an $M$-slab could percolate quickly.

\begin{definition}
A \emph{wave} is a sequence of droplets $(D_1,\dots,D_k)$, where $D_i=[(a_i,b_i),(c_i,d_i)]$ for each $i$, satisfying the following conditions:
\begin{enumerate}
\item the droplets are disjoint: $D_i\cap D_j=\emptyset$ if $i\neq j$;
\item the droplets are closed: $[D_1\cup\dots\cup D_k]=D_1\cup\dots\cup D_k$;
\item $b_i<b_{i+1}\leq d_i+2<d_{i+1}+2$ for $i=1,\dots,k-1$.
\end{enumerate}
\end{definition}

The \emph{height} of the wave, $h(W)$, is defined to be $d_k-b_1+1$. The three conditions of a wave imply that consecutive droplets do not overlap horizontally, so if $x\in D_i$ and $y\in D_{i+1}$ then $x\neq y$. Thus the quantity
\[
t_i = \min\{|a_{i+1}-c_i|,|a_i-c_{i+1}|\} = \max\{a_{i+1}-c_i,a_i-c_{i+1}\}
\]
is the horizontal distance between droplets $D_i$ and $D_{i+1}$. The \emph{time} of the wave, $t(W)$, is
\[
t(W) := \sum_{i=1}^{k-1} (t_i-1).
\]
The concept of the time of a wave is important. If the set of initially infected sites consists of the union of the row of sites immediately below $D_1$ (extending as far as necessary) and $D_1\cup\dots\cup D_{k-1}$, and if $D_k$ is a single site, then $t(W)$ is a lower bound for the time it takes $D_k$ to become infected.

A wave $W=(D_1,\dots,D_k)$ inside a droplet $D=[(a,b),(c,d)]$ is an \emph{up-wave} if $b_1=b$, and a \emph{down-wave} if $d_k=d$. Although the property of being an up- or down-wave depends on the parent droplet $D$, we shall rarely make reference to this.

The \emph{upper crest} of a wave $W=(D_1,\dots,D_k)$ in a droplet $D$ is the set
\[
\crest^+(W) = \{(x_1,x_2)\in D: b_k-1\leq x_2\leq d_k\} \setminus D_k,
\]
and similarly the \emph{lower crest} is the set
\[
\crest^-(W) = \{(x_1,x_2)\in D: b_1\leq x_2\leq d_1+1\} \setminus D_1.
\]
If $x=(x_1,x_2)$ is in the upper crest of $W$, then the \emph{upper $W$-time of $x$} is defined to be
\[
t^+(x,W) = t(W) + \min\{|x_1-c_k|,|a_k-x_1|\}
\]
if $x_1\notin [a_k,c_k]$, and $0$ otherwise, while if $x$ is in the lower crest of $W$, then the \emph{lower $W$-time of $x$} is defined to be
\[
t^-(x,W) = t(W) + \min\{|x_1-c_1|,|a_1-x_1|\}
\]
if $x_1\notin [a_1,c_1]$, and $0$ otherwise.

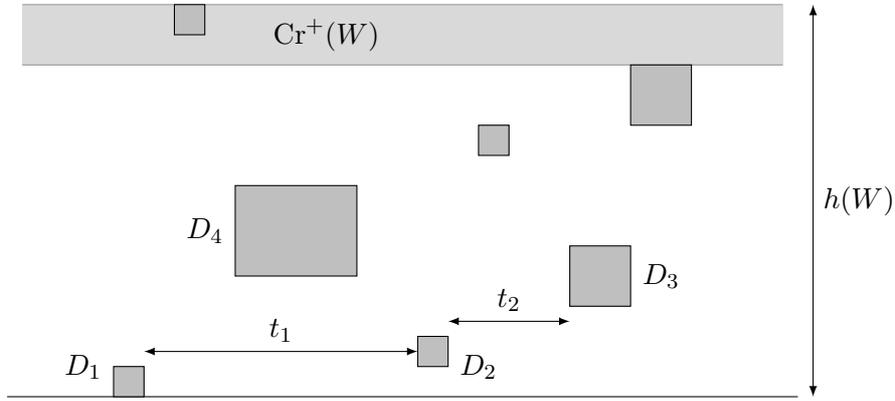
\begin{figure}[ht]
  \centering
  \begin{tikzpicture}[scale=0.4,>=latex]
    \tikzstyle{drop}=[fill=gray!50];
    \fill [gray!30] (0,11) rectangle (25,13);
    \draw [gray!80] (0,11) -- (25,11) (0,13) -- (25,13);
    \draw (-0.5,0) -- (25.5,0);
    \draw [drop] (3,0) rectangle (4,1);
    \draw [drop] (13,1) rectangle (14,2);
    \draw [drop] (18,3) rectangle (20,5);
    \draw [drop] (7,4) rectangle (11,7);
    \draw [drop] (15,8) rectangle (16,9);
    \draw [drop] (20,9) rectangle (22,11);
    \draw [drop] (5,12) rectangle (6,13);
    \draw [<->] (4,1.5) -- node [above] {$t_1$} (13,1.5);
    \draw [<->] (14,2.5) -- node [above] {$t_2$} (18,2.5);
    \draw [<->] (26,0) -- node [right] {$h(W)$} (26,13);
    \node at (2,1) {$D_1$};
    \node at (15,1) {$D_2$};
    \node at (21,4) {$D_3$};
    \node at (6,5.5) {$D_4$};
    \node at (10,12) {$\crest^+(W)$};
  \end{tikzpicture}
  \caption{An example of an up-wave $W$.}\label{fi:wave}
\end{figure}

Let $D=[(a,b),(c,d)]$ be a droplet and let
\[
D_0=[D\cap A] \cup \big([(a-1,b-1),(c+1,d+1)]\setminus D\big).
\]
Thus, $D_0$ is the union of the closure of $A$ restricted to $D$ and the horizontal and vertical lines adjacent to the edges of $D$.

\begin{definition}
For $t\geq 0$, the \emph{$t$-flood of $D$}, which we write as $[[D]]_t$, is defined to be the set $[D_0]_t\cap D$. The \emph{flood time} of $x\in D$ is the minimal $t$ such that $x\in[[D]]_t$ (which is well defined, because $D\subset[D_0]$). The flood time of $D$ itself is defined to be the maximum of the flood times of the sites belonging to $D$, or equivalently, it is the minimal $t$ such that $[[D]]_t=D$.
\end{definition}

It is easy to see that
\begin{equation}\label{eq:tflood}
A_t \subset D^c \cup [[D]]_t.
\end{equation}
The reason for this is that, firstly, $A \subset D^c \cup [[D]]_0$, because $D^c\cap A\subset D^c$ and $D\cap A\subset [[D]]_0$, and then \eqref{eq:tflood} follows because $D^c \cup [[D]]_t = [D^c\cup [[D]]_0]_t$. This simple observation means that we can bound from below the percolation time of the grid by the flood time of any given droplet.

Given a site $x=(x_1,x_2)\in D$, the \emph{width of $x$}, $w(x)$, is the minimum horizontal distance from $x$ to the exterior of $D$; specifically, $w(x) = \min\{c-x_1,x_1-a\}+1$. Similarly, the \emph{height of $x$}, $h(x)$, is the minimum vertical distance from $x$ to the exterior of $D$; thus, $h(x) = \min\{d-x_2,x_2-b\}+1$. The \emph{down-wake of $x$} is the set
\[
\{y=(y_1,y_2)\in D: |y_1-x_1|+y_2 \leq x_2\}.
\]
One may think of the down-wake of $x$ as the set of sites in the $45^\circ$ pyramid below $x$, with $x$ at the apex. The \emph{up-}, \emph{left-} and \emph{right-wake of $x$} are similarly defined. An easy induction shows that if $t$ is the flood time of $x$ and $t$ is strictly positive, then one of the four wakes of $x$ is wholly contained in the $t$-flood of $D$.

\begin{figure}[ht]
  \centering
  \begin{tikzpicture}[scale=0.4,>=latex]
    \tikzstyle{drop}=[fill=gray!50];
    \tikzstyle{flood}=[fill=gray!30];
    \path [flood] (0,0) -- (9,0) -- (9,1) -- (7,1) -- (7,2) -- (4,2) -- (4,3) -- (3,3) -- (3,4) -- (2,4) -- (2,5) -- (1,5) -- (1,6) -- (0,6) -- (0,0);
    \path [flood] (0,15) -- (6,15) -- (6,14) -- (5,14) -- (5,13) -- (4,13) -- (4,12) -- (3,12) -- (3,11) -- (2,11) -- (2,10) -- (1,10) -- (1,9) -- (0,9) -- (0,15);
    \path [flood] (25,0) -- (19,0) -- (19,1) -- (20,1) -- (20,2) -- (21,2) -- (21,3) -- (22,3) -- (22,4) -- (23,4) -- (23,5) -- (24,5) -- (24,6) -- (25,6) -- (25,0);
    \path [flood] (25,15) -- (19,15) -- (19,14) -- (20,14) -- (20,13) -- (21,13) -- (21,12) -- (22,12) -- (22,11) -- (23,11) -- (23,10) -- (24,10) -- (24,9) -- (25,9) -- (25,15);
    \path [flood] (10,0) rectangle (16,1);
    \path [flood] (7,15) -- (7,14) -- (8,14) -- (8,13) -- (17,13) -- (17,14) -- (18,14) -- (18,15) -- (7,15);
    \draw (0,0) rectangle (25,15);
    \draw (0,0) -- (9,0) -- (9,1) -- (7,1) -- (7,2) -- (4,2) -- (4,3) -- (3,3) -- (3,4) -- (2,4) -- (2,5) -- (1,5) -- (1,6) -- (0,6) -- (0,0);
    \draw (0,15) -- (6,15) -- (6,14) -- (5,14) -- (5,13) -- (4,13) -- (4,12) -- (3,12) -- (3,11) -- (2,11) -- (2,10) -- (1,10) -- (1,9) -- (0,9) -- (0,15);
    \draw (25,0) -- (19,0) -- (19,1) -- (20,1) -- (20,2) -- (21,2) -- (21,3) -- (22,3) -- (22,4) -- (23,4) -- (23,5) -- (24,5) -- (24,6) -- (25,6) -- (25,0);
    \draw (25,15) -- (19,15) -- (19,14) -- (20,14) -- (20,13) -- (21,13) -- (21,12) -- (22,12) -- (22,11) -- (23,11) -- (23,10) -- (24,10) -- (24,9) -- (25,9) -- (25,15);
    \draw (10,0) rectangle (16,1);
    \draw (7,15) -- (7,14) -- (8,14) -- (8,13) -- (17,13) -- (17,14) -- (18,14) -- (18,15) -- (7,15);
    \draw [drop] (2,1) rectangle (3,2);
    \draw [drop] (9,0) rectangle (10,1);
    \draw [drop] (18,3) rectangle (19,4);
    \draw [drop] (11,5) rectangle (12,6);
    \draw [drop] (4,8) rectangle (7,10);
    \draw [drop] (12,13) rectangle (13,14);
    \draw [drop] (20,11) rectangle (22,13);
    \draw [drop] (22,4) rectangle (23,5);
    \draw [drop] (16,7) rectangle (17,9);

  \end{tikzpicture}
  \caption{The set $[[D]]_0=[D\cap A]$ is shown by the dark rectangles. The $6$-flood $[[D]]_6$ is the union of the light and dark areas.}\label{fi:flood}
\end{figure}
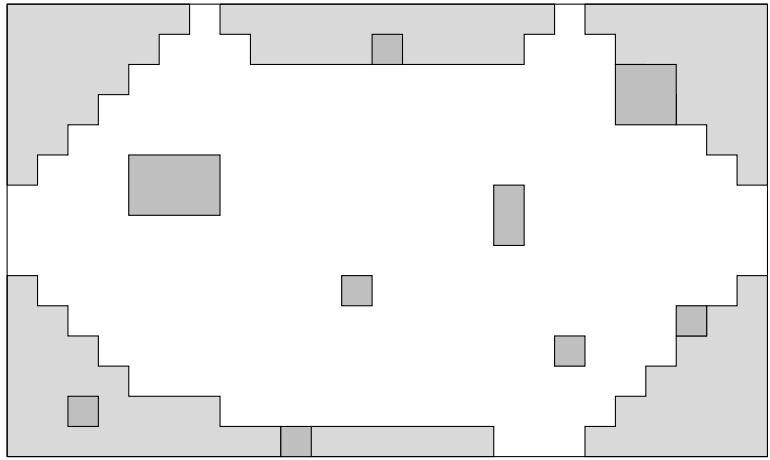

\begin{lemma}\label{le:wave}
Let $D$ be a droplet. Let $x$ be a site in $D$ with strictly positive flood time $t$, and suppose that $t<w(x)$. Then $[[D]]_0$ contains an up- or down-wave with height at least $h(x)$ and time at most $t$.
\end{lemma}

\begin{proof}
The proof is by induction on the flood time $t$. We strengthen the claim slightly by proving, under the same conditions, that $[[D]]_0$ contains a wave $W$ with height at least $h(x)$, such that \emph{either} $W$ is an up-wave, $x$ is in the upper crest of $W$, and $t^+(x,W)\leq t$, \emph{or} $W$ is a down-wave, $x$ is in the lower crest of $W$, and $t^-(x,W)\leq t$.

If $x$ has flood time $1$ then $x$ lies on the top or bottom edge of $D$, is not at one of the corners (since $w(x)>1$), and is adjacent to (exactly one) site in $[[D]]_0$, which must also lie on the top or bottom edge of $D$, so the claim is true.

Suppose the claim is true for all sites with flood time $t-1$. Since $x$ has flood time $t$, at least one of its four neighbours must have flood time $t-1$. Such a neighbour $y$ has width $w(y)\geq w(x)-1$, so the induction hypothesis applies and without loss of generality there is an up-wave $W$ with height at least $h(y)$ such that $y$ is in the upper crest of $W$ and the upper $W$-time of $y$ is at most $t-1$. Observe that if $h(W)$ is at least $h(x)$ then the same wave $W$ satisfies the conditions of the claim for $x$, since then $x$ is in the upper crest of $W$ and $t^+(x,W)\leq t$. So we may assume that $h(x)>h(W)$. This means that $y$ cannot be equal to $x\pm e_1$, since that would imply $h(x)=h(y)\leq h(W)$.

If both $x+e_2$ and $x-e_2$ have flood time at most $t-1$, and their associated waves are $W$ and $W'$ respectively, then $W$ must be a down-wave and $W'$ an up-wave, and one of $h(W)$ or $h(W')$ must be greater than $h(x)$, which is again a contradiction.

We are left with the case $y=x-e_2$ (if $y=x+e_2$ then $W$ is a down-wave and the argument is similar), $h(y)=h(W)$, and each of $x+e_1$, $x-e_1$ and $x+e_2$ either belongs to $[[D]]_0$ or has flood time at least $t$. In fact at least one of $x+e_1$, $x-e_1$ and $x+e_2$ must belong to $[[D]]_0$, because otherwise at most one neighbour of $x$ would be in $[[D]]_{t-1}$, so $x$ would not be infected at time $t$.

\emph{Case 1.} Suppose $x-e_1\in[[D]]_0$; the case $x+e_1\in[[D]]_0$ is treated in the same way. Let $D'=[(a,b),(c,d)]$ be the maximal droplet in $[[D]]_0$ that contains $x-e_i$. If $D'$ is one of the droplets in the wave $W$ then it must be that $h(W)\geq h(x)$ and $t(x,W)\leq t$, so $W$ satisfies the conditions for $x$. So we may assume instead that $D'$ is not one of the droplets in $W$. We shall show there exists $j\leq k$ such that $W'=(D_1,\dots,D_j,D')$ is an up-wave, $h(W')\geq h(x)$, and $t^+(x,W')\leq t$.

First we need to establish that $W'$ satisfies the three conditions of a wave. The first two (that the droplets are disjoint and form a closed set) are trivially satisfied for all $j$. Since $W$ is wave, the third condition for a wave is satisfied for all $1\leq i\leq j-1$, so it remains to show that $b_j<b\leq d_j+2<d+2$. The third of these inequalities is satisfied by all $j$ because $d_j\leq h(W)<h(x)\leq d$. The second inequality is satisified when $j=k$, because $b\leq h(x)=h(y)+1=h(W)+1=d_k+1$. Choose $j$ to be minimal such that the second inequality is satisfied. If $j=1$ then we take $W'=(D')$ if $b=1$ and $W'=(D_1,D')$ if $b>1$. Otherwise, if $j>1$, then $b_j\leq d_{j-1}+2$ because $W$ is a wave, and we have $d_{j-1}+2<b$ by the minimality of $j$, so $b_j<b$, and therefore $j$ satisfies the first inequality. This proves that $W'$ is a wave, and hence, since $W$ is an up-wave, that $W'$ is also an up-wave. The inequality $h(W')\geq h(x)$ follows immediately because $x\in D'$.

It remains to show that $t^+(x,W')\leq t$. For this, it is enough to have $t(W')\leq t-1$, because $x$ is horizontally adjacent to $D'$. We have
\begin{align}\label{eq:tplusineq}
t(W') &\leq t(W) + \min\{|a-c_j|,|a_j-c|\} - 1 \notag \\
&\leq t(W) + \min\{|y_1-c_j|,|a_j-y_1|\} \notag \\
&\leq t^+(y,W),
\end{align}
which is at most $t-1$, as required. This completes the case in which $x\pm e_1\in[[D]]_0$.

\emph{Case 2.} Now suppose $x+e_2\in[[D]]_0$. As before, let $D'=[(a,b),(c,d)]$ be the maximal droplet in $[[D]]_0$ that contains $x-e_i$. Observe that $D'$ cannot be one of the droplets in $D$, because $b=h(x)+1=h(y)+2=h(W)+2$. We shall show that $W'=(D_1,\dots,D_k,D')$ is an up-wave, $h(W')\geq h(x)$, and $t^+(x,W')\leq t$. That $W'$ is a wave is clear, because the first two properties of a wave (that the droplets are disjoint and form a closed set) are again trivially satisfied, and the third condition, that $b_k<b\leq d_k+2<d+2$, is also satisfied, because we have just observed that $b=d_k+2$. Given that $W'$ is a wave, it is automatically an up-wave, and the inequality $h(W')\geq h(x)$ is also clear, because $x\in D'$. Our final task, then, is to show that $t^+(x,W')\leq t$. Since $x_1\in[a,c]$, the condition is equivalent to $t(W')\leq t$. But now the calculation is the same as in \eqref{eq:tplusineq}, which completes the proof of this case, and also the proof of the lemma.
\end{proof}

\subsection{Subcriticality and restrictions of waves}\label{se:restrictions}

It turns out to be too difficult to do calculations with waves considered in the level of generality we have so far permitted, because there are too many choices for the dimensions of the $D_i$. The purpose of this short section is to introduce new, related structures, which allow only three types of droplet. The reason we are able to restrict the number of types of droplet so strongly is because we only ever consider waves in subcritical regions of the grid, so there are no large internally spanned droplets. The advantage of introducing the simplified strctures is that calculations involving these structures are much simpler. The disadvantage is that the new structures are not necessarily themselves waves, because the droplets may overlap. However, this causes only very minor complications.

First, let us say that an $M$-slab $D$ is \emph{subcritical} if the largest internally spanned droplet in $D$ has semi-perimeter at most $\gamma$. All of the $M$-slabs we shall consider will be subcritical. Similarly we say that a wave $W=(D_1,\dots,D_k)$ is subcritical if $\phi(D_i)\leq\gamma$ for $1\leq i\leq k$.

Let $\sigma$ be a fixed positive integer to be specified later. The \emph{$(1,\sigma,\gamma)$-restriction} of a subcritical wave $W=(D_1,\dots,D_k)$ is the sequence of droplets $(D_1',\dots,D_{k'}')$ obtained by applying the following algorithm to $W$.
\begin{enumerate}
\item Let $i$ be minimal such that $\sigma+2\leq\phi(D_i)\leq\gamma$, or if no such $i$ exists then move on to the next step. Replace $D_i=[(a,b),(c,d)]$ by the $\gamma$-cell $[(a,b),(a+\gamma-1,b+\gamma-1)]$. Remove from the sequence all droplets $D_j=[(a',b'),(c',d')]$ which are such that $b'\geq b$ and $d'\leq d$. Repeat this step until all droplets are either $\gamma$-cells or have semi-perimeter at most $\sigma+1$.
\item Let $i$ be minimal such that $3\leq\phi(D_i)\leq\sigma+1$, or if no such $i$ exists then stop. Replace $D_i=[(a,b),(c,d)]$ by the $\sigma$-cell $[(a,b),(a+\sigma-1,b+\sigma-1)]$. Remove from the sequence all droplets $D_j=[(a',b'),(c',d')]$ which are such that $b'\geq b$ and $d'\leq d$. Repeat this step until all droplets are either $\gamma$-cells, $\sigma$-cells, or consist of a single site, and then stop.
\end{enumerate}

Note that the definition implies that $D_i\subset D_i'$ for all $i$. As mentioned above, the $(1,\sigma,\gamma)$-restriction of a wave is not necessarily a wave, because $D_i'\cap D_{i+1}'$ may be non-empty. However, the following lemma, which is merely an observation, is the only disjointness property we need.

\begin{lemma}\label{le:isrestricted}
Let $W=(D_1,\dots,D_k)$ be a subcritical wave and let $W'=(D_1',\dots,D_{k'}')$ be the $(1,\sigma,\gamma)$-restriction of $W$. Suppose that the $D_i$ are all internally spanned. Then for each $1\leq j\leq k'$, the following holds. If $D_j'$ is a single site then it is internally spanned. If $D_j'$ is a $\sigma$-cell then it contains a droplet of semi-perimeter between $3$ and $\sigma+1$ which is internally spanned. If $D_j'$ is a $\gamma$-cell then it contains a droplet of semi-perimeter between $\sigma+2$ and $\gamma$ which is internally spanned. Furthermore, the internally spanned droplets associated with the $D_j'$ are disjoint. \qed
\end{lemma}

Let $W'=(D_1',\dots,D_{k'}')$ be the $(1,\sigma,\gamma)$ restriction of a wave $W$, and let $D_i'=[(a_i',b_i'),(c_i',d_i')]$ for $1\leq i\leq k'$. The \emph{height} of $W'$, like that of a wave, is defined to be $h(W')=d_{k'}'-b_1'+1$. The horizontal displacement between droplets $D_i'$ and $D_{i+1}'$ is
\[
t_i' = \max\{a_{i+1}'-c_i',a_i'-c_{i+1}',1\}.
\]
The time of $W'$ is then defined to be
\[
t(W') = \sum_{i=1}^{k'} (t_i'-1).
\]

\begin{lemma}\label{le:htrestricted}
Let $W'$ be the $(1,\sigma,\gamma)$-restriction of a subcritical wave $W$. Then $h(W')\geq h(W)$ and $t(W')\leq t(W)$. \qed
\end{lemma}

\subsection{Slow percolation of subcritical $M$-slabs}\label{se:slowperc}

This is the part of the proof of the lower bound in the small $p$ regime where the main calculations occur. We show that the probability there exists an up-wave with height $h$ and time $t$ inside a subcritical $M$-slab is small provided $pt/h$ is at most a small constant. This corresponds to the intuition that infection should spread in towards the centre of a subcritical $M$-slab at rate $\Theta(p)$.

The \emph{anchor} of a wave or restricted wave $W=(D_1,\dots,D_k)$ is the ordered pair $(a,b)$, where $D_1=[(a,b),(c,d)]$ for some $c,d$. A wave is said to be \emph{anchored} if its anchor is fixed.

Let $V(h,t)$ be the event that there exists a subcritical wave $W=(D_1,\dots,D_k)$ with anchor at the origin such that $W$ has height exactly $h$ and time exactly $t$. Let $V_\Gamma(h,t)$ denote the event that $V(h,t)$ occurs and that the number of $\gamma$-cells in the $(1,\sigma,\gamma)$-restriction of $W$ is at least $hp/\gamma$, and let $V_{\Gamma^c}(h,t) = V(h,t) \cap V_\Gamma(h,t)^c$. The calculation to show that $P_p\big(V(h,t)\big)$ is small provided $pt/h$ is small is slightly different according to whether or not the wave $W$ contains a large number of $\gamma$-cells, so we separate the calculation into two parts.

\begin{lemma}\label{le:fewcrit}
Let $pt/h$ be sufficiently small and let $\sigma$ be sufficiently large. Then
\[
\P_p\big(V_{\Gamma^c}(h,t)\big) \leq h^3 e^{-h/\gamma}.
\]
\end{lemma}

\begin{proof}
Our first task is to find an upper bound for the number of possible $(1,\sigma,\gamma)$-restrictions of an anchored wave of a given height and time. To that end, let $W'=(D_1',\dots,D_{k'}')$ be a sequence of droplets such that there exists a wave of which $W'$ is the $(1,\sigma,\gamma)$-restriction. (The only reason we demand the existence of the wave is to limit the number of possibilities for the positions of the droplets $D_i'$.) Let $a$ be the number of single site droplets in $W'$, $b$ the number of $\sigma$-cells, and $c$ the number of $\gamma$-cells. Thus $k'=a+b+c$. Fix also the time $t$ of $W'$ and anchor $W'$ at $(0,0)$. The number of choices for the order of the three sizes of droplets is
\begin{equation}\label{eq:dropsizes}
\binom{a+b+c}{a,b,c} \leq 3^{a+b+c}.
\end{equation}
If $D_i'$ is a single site then there are two choices for the vertical displacement $b_i'-d_{i-1}'$; specifically, either $b_i'=d_{i-1}'+1$ or $b_i'=d_{i-1}'+2$. Similarly if $D_i'$ is a $\sigma$-cell there are $\sigma+1$ choices for $b_i'$, and if it is a $\gamma$-cell there are $\gamma+1$ choices. In total this gives
\begin{equation}\label{eq:dropvert}
2^a(\sigma+1)^b(\gamma+1)^c \leq C^{a+b+c}\sigma^b\gamma^c
\end{equation}
choices for the vertical displacements between the droplets, for some $C>0$. Next, recall that
\[
t(W') = \sum_{i=1}^{a+b+c} (t_i'-1),
\]
where $t_i'=\max\{a_{i+1}'-c_i',a_i'-c_{i+1}',1\}$. Thus there are $\binom{t}{a+b+c}$ choices for the $t_i'$ and a further $3^{a+b+c}$ choices for whether we have $a_{i+1}'-c_i'\geq 1$, or $a_i'-c_{i+1}'\geq 1$, or neither. If neither then the droplets overlap horizontally; in this case there are (crudely) at most $(3\sigma)^{2b}(3\gamma)^{2c}$ choices for $a_{i+1}'-a_i'$. (If the larger of the two cells is a $\gamma$-cell then there are at most $3\gamma$ choices for the displacement, while if it is a $\sigma$-cell then there are at most $3\sigma$ such choices. Furthermore, each cell can contribute to at most two displacements.) In total, there are at most
\begin{equation}\label{eq:drophoriz}
\binom{t}{a+b+c} 3^{a+b+c} (3\sigma)^{2b} (3\gamma)^{2c} \leq C^{a+b+c} \sigma^{2b} \gamma^{2c} \binom{t}{a+b+c}
\end{equation}
choices for the horizontal displacements $a_{i+1}'-a_i'$, for some $C>0$. Multiplying together the number of choices for the order of the sizes of the droplets \eqref{eq:dropsizes}, the vertical displacements \eqref{eq:dropvert}, and the horizontal displacements \eqref{eq:drophoriz}, we have that the number of possibilities for the sequence of droplets in the restricted wave $W'$ is at most
\[
C^{a+b+c} \sigma^{3b} \gamma^{3c} \binom{t}{a+b+c}
\]
for some constant $C>0$. Since $\sigma$ is also a constant and $\gamma=p^{-3}$, this is at most
\begin{equation}\label{eq:numwaves}
C^{a+b+c} p^{-9c} \binom{t}{a+b+c}.
\end{equation}

Our next task is to bound the probability that a given $W'$ is the $(1,\sigma,\gamma)$-restriction of a wave of internally spanned droplets. The probability that a $\sigma$-cell contains an internally spanned droplet of semi-perimeter between $3$ and $\sigma+1$ is $O(p^2)$, because the internally spanned droplet must contain at least two initially infected sites, and there are a constant number of choices for their positions. The probability that a $\gamma$-cell contains an internally spanned droplet $D$ with semi-perimeter between $\sigma+2$ and $\gamma$ is $O(p^{\sigma/2-5})$. The reason is as follows. The internally spanned droplet $D$ must itself contain another internally spanned droplet $D'$ having semi-perimeter between $\sigma+2$ and $2(\sigma+2)$, by Lemma \ref{le:AL}. The probability that $D'$ is internally spanned is $O(p^{\sigma/2+1})$, by Lemma \ref{le:phi}, and there are $O(\gamma^2)$ choices for $D'$, so the probability that a $\gamma$-cell contains an internally spanned droplet of semi-perimeter between $\sigma+2$ and $\gamma$ is at most
\[
O\big( p^{\sigma/2+1} \gamma^2 \big) = O(p^{\sigma/2-5}).
\]
Hence, the probability that there exists a wave $W$ of internally spanned droplets such that $W'$ is the $(1,\sigma,\gamma)$-restriction of $W$ is at most
\begin{equation}\label{eq:probwave}
p^a (Cp^2)^b (Cp^{\sigma/2-5})^c = C^{a+b} p^{a+2b+(\sigma/2-5)c}
\end{equation}
for some positive constant $C$.

The wave $W$ has height $h$ and time $t$, so its $(1,\sigma,\gamma)$-restriction has height at least $h$ and time at most $t$, by Lemma \ref{le:htrestricted}. Hence we have $2a+(\sigma+1)b+(\gamma+1)c\geq h$, or very crudely, $a+\sigma b+\gamma c\geq h/2$. Combining this with the bound in \eqref{eq:numwaves} for number of possibilities for $W'$ and the bound in \eqref{eq:probwave} for the probability that a given $W'$ is the $(1,\sigma,\gamma)$-restriction of a wave of internally spanned droplets gives
\begin{equation}\label{eq:Vht}
\P_p\big( V_{\Gamma^c}(h,t) \big) \leq \sum_{\substack{0\leq a,b,c\leq h \\ a+\sigma b+\gamma c\geq h/2 \\ c\leq hp/\gamma}} C^{a+b+c} \binom{t}{a+b+c} p^{a+2b+(\sigma/2-14)c}.
\end{equation}
Hence, by Stirling's formula and the assumption that $\sigma$ is sufficiently large, we have
\begin{equation}\label{eq:nasty}
\P_p\big( V_{\Gamma^c}(h,t) \big) \leq \sum_{\substack{0\leq a,b,c\leq h \\ a+\sigma b+\gamma c\geq h/2 \\ c\leq hp/\gamma}} \left(\frac{Ct}{a+b+c}\right)^{a+b+c} p^{a+2b+3c}.
\end{equation}
The calculation needed in order to bound the right-hand side of \eqref{eq:nasty} is routine and unenlightening, so it is deferred until Lemma \ref{le:calc} in the Appendix. (It is in this calculation that use the assumption that $c<hp/\gamma$.) The calculation implies that
\begin{equation}\label{eq:VGammac}
\P_p\big( V_{\Gamma^c}(h,t) \big) \leq \sum_{\substack{0\leq a,b,c\leq h \\ a+\sigma b+\gamma c\geq h/2 \\ c\leq hp/\gamma}} e^{-(a+b+c)}.
\end{equation}
Thus, using the observation that $a+b+c\geq h/\gamma$, we have
\[
\P_p\big( V_{\Gamma^c}(h,t) \big) \leq h^3 e^{-h/\gamma},
\]
which is the inequality claimed in the statement of the lemma.
\end{proof}

\begin{lemma}\label{le:manycrit}
Let $pt/h=O(1)$, let $p$ be sufficiently small, and suppose that $\sigma\geq 36$. Then
\[
\P_p\big(V_\Gamma(h,t)\big) \leq e^{-hp^4}.
\]
\end{lemma}

\begin{proof}
Suppose $V_\Gamma(h,t)$ occurs. Let $W$ be the associated subcritical wave and let $W'=(D_1',\dots,D_{k'}')$ be its $(1,\sigma,\gamma)$-restriction, with $D_i'=[(a_i',b_i'),(c_i',d_i')]$ for $1\leq i\leq k'$. Recall that the number of $\gamma$-cells in $W'$ must be at least $hp/\gamma$, by the definition of the event $V_\Gamma(h,t)$. Let the $\gamma$-cells in $W'$ be $D_{i_1}'\dots,D_{i_c}'$, where $c\geq hp/\gamma$. We would like to bound the number of possible positions for these $\gamma$-cells. First, no two droplets can have the same $b_i$ coordinate, so there are at most $\binom{h}{c}$ choices for $b_{i_1}',\dots,b_{i_c}'$. Next, recall that the time of $W'$ is defined to be
\[
t(W') = \sum_{i=1}^{k'-1} (t_i'-1),
\]
where $t_i' = \max\{ a_{i+1}'-c_i',a_i'-c_{i+1}',1 \}$ is the horizontal offset between droplets $D_i'$ and $D_{i+1}'$. Suppose $D_{i_{j+1}}'$ lies to the right of $D_{i_j}'$, so $a_{i_{j+1}}'-c_{i_j}'\geq 1$. Observe that $c_l'-a_l'\leq\sigma$ whenever $D_l'$ is not a $\gamma$-cell, and hence
\[
a_{i_{j+1}}'-c_{i_j}' = \sum_{l=i_j}^{i_{j+1}-1} (a_{l+1}'-c_l') + \sum_{l=i_j+1}^{i_{j+1}-1} (c_l'-a_l') \leq \sum_{l=i_j}^{i_{j+1}-1} t_l' + (i_{j+1}-i_j-1)\sigma.
\]
The same inequality holds for $a_{i_j}'-c_{i_{j+1}}'$ if $D_{i_{j+1}}'$ lies to the left of $D_{i_j}'$, so we have
\[
\max\{ a_{i_{j+1}}'-c_{i_j}',a_{i_j}'-c_{i_{j+1}}',0 \} \leq \sum_{l=i_j}^{i_{j+1}-1} t_l' + (i_{j+1}-i_j-1)\sigma.
\]
Summing over $\gamma$-cells we obtain
\[
\sum_{j=1}^c \max\{a_{i_{j+1}}'-c_{i_j}',a_{i_j}'-c_{i_{j+1}}',0\} \leq \sum_{i=1}^{k'} t_i' + (k'-c)\sigma = (t(W')+k') + (k'-c)\sigma = O(t),
\]
where here we have used $t(W')\leq t(W)=t$ from Lemma \ref{le:htrestricted} and $k'\leq t(W')$. Hence there are at most $\binom{O(t)}{c}$ choices for the absolute values of the non-zero horizontal offsets. There are a further $3^c$ choices for whether $a_{i_{j+1}}'-c_{i_j}'\geq 1$, or $a_{i_j}'-c_{i_{j+1}}'\geq 1$, or neither. When the offset is zero, there are at most $3\gamma$ choices for the value of $a_{i_{j+1}}'-a_{i_j}'$. The differences $a_{i_{j+1}}'-a_{i_j}'$ together with the $b_{i_j}'$ uniquely define the positions of the $\gamma$-cells. It follows from these observations that there are at most
\begin{equation}\label{eq:gammachoices}
\binom{h}{c} \binom{O(t)}{c} (9\gamma)^c 
\end{equation}
choices for the $\gamma$-cells.

By the same calculation as in the proof of Lemma \ref{le:fewcrit}, the probability that a $\gamma$-cell contains an internally spanned droplet with semi-perimeter between $\sigma+2$ and $\gamma$ is $O(p^{\sigma/2-5})$, and by Lemma \ref{le:isrestricted}, each $\gamma$-cell must contain such an internally spanned droplet. Combined with the bound for the number of choices of $\gamma$-cells in \eqref{eq:gammachoices}, this proves that
\[
\P_p\big( V_\Gamma(h,t) \big) \leq \sum_{c\geq hp/\gamma} \binom{h}{c} \binom{O(t)}{c} (C\gamma p^{\sigma/2-5})^c.
\]
Using Stirling's formula, we have
\[
\P_p\big( V_\Gamma(h,t) \big) \leq \sum_{c\geq hp/\gamma} \left(\frac{Cht\gamma p^{\sigma/2-5}}{c^2}\right)^c \leq \sum_{c\geq hp/\gamma} \left(C(pt/h)\gamma^3 p^{\sigma/2-8}\right)^c
\]
for some new constant $C$. Now, $pt/h=O(1)$, $\gamma=p^{-3}$, and $\sigma\geq 36$, so
\[
\P_p\big( V_\Gamma(h,t) \big) \leq \sum_{c\geq hp/\gamma} (Cp)^c \leq e^{-hp/\gamma},
\]
again for a new $C$, provided $p$ is sufficiently small.
\end{proof}

An $M$-slab $D$ is defined to be \emph{slow} if $[[D]]_{cM/p} \neq D$, where $c<1/2$ is a small positive constant, and otherwise it is \emph{fast}. We write $F(D)$ for the event that $D$ is fast.

In the following lemma and in the proof of Theorem \ref{th:small} we shall need to use Harris's Lemma \cite{Harris} (later generalized by Fortuin, Kasteleyn, and Ginibre \cite{FKG}, and now better known as the FKG inequality) to bound the probabilities of certain intersections of non-independent increasing and decreasing events. The definitions of \emph{increasing} and \emph{decreasing} are the usual percolation-theoretic definitions: an event $E$ is increasing if for all pairs of configurations $\omega$, $\omega'\in\{0,1\}^{[n]^2}$ such that $\omega\subset\omega'$, the implication $\omega\in F \; \Rightarrow \; \omega'\in F$ holds; it is \emph{decreasing} if the converse implication $\omega'\in F \; \Rightarrow \; \omega\in F$ holds. Harris's Lemma is as follows.

\begin{lemma}\label{le:harris}
Let $E$ and $F$ be increasing events and let $G$ be a decreasing event. Then
\[
\P_p(E\cap F) \geq \P_p(E) \cap \P_p(F)
\]
and
\[
\P_p(E\cap G) \leq \P_p(E) \cap \P_p(G) \tag*{\qed}
\]
\end{lemma}

\begin{lemma}\label{le:subandfast}
Let $D$ be an $M$-slab and let $p$ be sufficiently small. Then
\[
\P\big( F(D) \given \Gamma(D)^c \big) \leq \frac{1}{2}.
\]
\end{lemma}

\begin{proof}
Suppose an $M$-slab $D$ is subcritical and fast. There must exist a site $x$ with strictly positive flood time $t\leq\tau:=cM/p$ such that $x$ has $\ell_1$ distance at most $\gamma$ from the centre of $D$. Certainly $t<w(x)$, so Lemma \ref{le:wave} implies that $[[D]]_0$ contains an up- or down-wave $W=(D_1,\dots,D_k)$ with height at least $h(x)\geq M/4$ and time at most $t$. There are $M/p$ choices for the anchor of $W$ and two choices for whether it is an up- or down-wave. Hence
\[
\P_p\big(F(D) \given \Gamma(D)^c\big) \leq 2 (M/p) \sum_{t\leq\tau} \sum_{M/4\leq h\leq M} \Big( \P_p\big(V_\Gamma(h,t)\given\Gamma(D)^c\big) + \P_p\big(V_{\Gamma^c}(h,t)\given\Gamma(D)^c\big) \Big).
\]

The events $V_\Gamma(h,t)$ and $V_{\Gamma^c}(h,t)$ are increasing for all $h$ and $t$, while the event $\Gamma(D)^c$ is decreasing. So Harris's Lemma implies that
\[
\P_p\big(V_\Gamma(h,t)\given\Gamma(D)^c\big) = \frac{\P_p\big(V_\Gamma(h,t)\cap\Gamma(D)^c\big)}{\P_p\big(\Gamma(D)^c\big)} \leq \frac{\P_p\big(V_\Gamma(h,t)\big)\P_p\big(\Gamma(D)^c\big)}{\P_p\big(\Gamma(D)^c\big)} = \P_p\big(V_\Gamma(h,t)\big),
\]
and similarly $\P_p\big(V_{\Gamma^c}(h,t)\given\Gamma(D)^c\big) \leq \P_p\big(V_{\Gamma^c}(h,t)\big)$, in both cases for all $h$ and $t$. So we have
\[
\P_p\big(F(D) \given \Gamma(D)^c\big) \leq 2 (M/p) \sum_{t\leq\tau} \sum_{M/4\leq h\leq M} \Big( \P_p\big(V_\Gamma(h,t)\big) + \P_p\big(V_{\Gamma^c}(h,t)\big) \Big).
\]

Now we use the inequality for $\P_p(V_{\Gamma^c}(h,t))$ from Lemma \ref{le:fewcrit} and the inequality for $\P_p(V_\Gamma(h,t))$ from Lemma \ref{le:manycrit}. These give
\begin{align}\label{eq:subandfast}
\P_p\big(F(D)\given\Gamma(D)^c\big) &\leq 2 (M/p) \sum_{t\leq\tau} \sum_{M/4\leq h\leq M} \big( h^3 e^{-hp^3} + e^{-hp^4} \big) \notag \\
&= O\left( M^6 p^{-2} \exp(-Mp^4 / 4) \right).
\end{align}

Recall that $M=A\sqrt{p\log(n/K)}K$ and $K=\exp\big(\mu(p)/p\big)$. Now, $1/p \ll M$ and $\log M = O(1/p) \ll Mp^4$, so we can reduce \eqref{eq:subandfast} to
\[
\P_p\big(F(D)\given\Gamma(D)^c\big) \leq \exp(-Mp^4 / 8).
\]
Now, we have
\[
Mp^4 = p^{7/2} (\log n)^{1/2} K(p) = \exp\left( \frac{\lambda}{p} + \frac{\log\log n}{2} - o\left(\frac{1}{p}\right) \right) \to \infty.
\]
so $\P_p\big(F(D)\given\Gamma(D)^c\big) \leq 1/2$ if $p$ is sufficiently small.
\end{proof}

\subsection{Proof of the lower bound in Theorem \ref{th:small}}\label{se:pflowersmall}

\begin{proof}[Proof of the lower bound in Theorem \ref{th:small}]
As remarked in the introduction to Section \ref{se:lowersmall}, the bound $T\geq c(\log n)/p$ (with high probability) follows from Lemma \ref{le:lb}, so we only have to prove that
\[
T \geq c\sqrt{\frac{\log(n/K)}{p}}K
\]
with high probability, for some constant $c>0$. (In fact, we shall prove this statement with the same constant $c$ as in definition of a slow $M$-slab.)

Our proof will show that $[n]^2$ contains a slow $M$-slab with high probability. This will be sufficient, because the percolation time of $[n]^2$ is certainly at least the flood time of any given $M$-slab, and an $M$-slab was defined to be slow if its flood time was at least $cL/p$.

Let $D$ be an $M$-slab. We have
\begin{equation}\label{eq:probfast}
\P_p\big(F(D)^c\big) \geq \P_p\big(F(D)^c\cap\Gamma(D)^c\big) = \P_p\big(F(D)^c\given\Gamma(D)^c\big) \P_p\big(\Gamma(D)^c\big).
\end{equation}
The first probability, that $D$ is slow given that it is subcritical, is at least $1/2$ by Lemma \ref{le:subandfast}. So it remains for us to bound the second probability, that $D$ is subcritical.

Tile $D$ with $K$-cells $D_1,D_2,\dots,D_k$ so that neighbouring $K$-cells overlap by a distance of $2\gamma$; this ensures that any critical droplet in $D$ is entirely contained in at least one of the $K$-cells. The number of $K$-cells is
\[
k = \left(1+\frac{2\gamma}{K}\right)^2 \frac{M^2/p}{K^2} = (1+o(1))\log\left(\frac{n}{K}\right)
\]

The events $\Gamma(D_1)^c,\dots,\Gamma(D_k)^c$ are decreasing, so we can apply Harris's Lemma (Lemma \ref{le:harris}) to show that the probability $D$ is subcritical is
\[
\P_p\big(\Gamma(D)^c\big) = \P_p\big(\Gamma(D_1)^c\cap\dots\cap \Gamma(D_k)^c\big) \geq \P_p\big(\Gamma(D_1)^c\big) \dots \P_p\big(\Gamma(D_k)^c\big).
\]
By Lemma \ref{le:mu}, the probability any one of the $K$-cells is critical is at most $3/4$, so we have $\P_p\big(\Gamma(D)^c\big) \geq 4^{-k}$. Together with \eqref{eq:probfast} and the inequality $\P\big(F(D)^c\given\Gamma^c(D)\big)\geq 1/2$ from Lemma \ref{le:subandfast}, this implies that
\[
\P_p\big(F(D)^c\big) \geq 4^{-k-1} = \exp\big(-(\log 4 + o(1))\log(n/K)\big).
\]

Let $E$ be the event that every $M$-slab contained in $[n]^2$ is fast. By dividing $[n]^2$ into disjoint $M$-slabs, we see that the probability there does not exist a slow $M$-slab is
\[
\P_p(E) \leq \Big(1 - \exp\big(-(\log 4 + o(1))\log(n/K)\big)\Big)^{\frac{n^2p}{M^2}} \leq \exp\left(\frac{-e^{-(\log 2 + o(1))\log(n/K)}n^2}{\log(n/K) K^2}\right).
\]
Hence
\[
\P_p(E) \leq \exp\Big(-\exp\big( (2 - \log 4 + o(1))\log(n/K) - \log\log(n/K) \big)\Big).
\]
Recall that $K=\exp\big(\mu(p)/p\big)$, where $\mu(p)=\lambda+o(1)$. This, combined with the condition $\liminf p\log n\geq (1+\epsilon)\lambda$ for some $\epsilon>0$, implies that $\log(n/K)=\Omega(\log n)$ provided $p$ is sufficiently small, and in particular this means that $\P_p(E)=o(1)$. This completes the proof of Theorem \ref{th:small}.
\end{proof}

\section{Open Problems}

We begin this final section by asking to what extent the condition $\liminf p\log n>\lambda$ in the statement of Theorem \ref{th:small} can be weakened. In particular, we ask whether the theorem holds under the weaker assumption that
\[
\liminf_{n\to\infty} \, (\log n)^{3/2}\big(p-p_c([n]^2,2)\big) > 0.
\]
Here, the $(\log n)^{3/2}$ renormalizing factor is motivated by the result of Morris \cite{MorrisSharp}, which states that
\[
p_c([n]^2,2) = \frac{\lambda}{\log n} - \frac{\Theta(1)}{(\log n)^{3/2}}.
\]
% on the sharp form of the critical probability $p_c([n]^2,2)$ for details), does the theorem hold if
%\begin{equation}
%p > \frac{\lambda}{\log n} - \frac{c}{(\log n)^{3/2}},
%\end{equation}
%for a sufficiently small constant $c$?
Weakening the assumption still further, we ask the following:

\begin{question}
Is the percolation time concentrated if we assume only that
\[
\P_p\big(\text{a $p$-random subset of $[n]^2$ percolates}) = 1-o(1)?
\]
\end{question}

Note that, as observed in the introduction, we do not expect the percolation time to be concentrated when the probability of percolation is bounded away from 0 and 1. However, it is interesting to ask what happens when $p$ is in the subcritical regime, \emph{conditioned on percolation occurring}.

\begin{question}
Let $p$ be such that $\limsup p\log n<\lambda$ (that is, $p$ is \emph{subcritical}) and let $A$ be a $p$-random subset of $[n]^2$ conditioned on percolation occurring. Is the percolation time of $A$ concentrated?
\end{question}

Finally, instead of taking a $p$-random subset, we select a uniformly random subset of fixed cardinality $m$. This allows us to ask much more precise questions, such as the following:

%we ask which $m\geq n$ maximizes the expected percolation time of a randomly chosen subset of $[n]^2$ of cardinality $m$, again conditioned on percolation occurring.

\begin{question}\label{qu:max}
Let $A$ be a subset of $[n]^2$ chosen uniformly at random from all percolating subsets of $[n]^2$ of size $m$. For which $m\geq n$ is the expected percolation time of $A$ maximized?
\end{question}

It is tempting to conjecture that the maximum should occur at $m=n$. (Note that $m=n$ is the minimum size of a percolating set.) However, recent results of Benevides and Przykucki \cite{BenPrz,BenPrz2} show that, somewhat surprisingly, in the deterministic setting, the maximum percolation time of \emph{any} percolating set, which is $\Theta(n^2)$, is achieved with an initial configuration of $cn$ sites, for some $c>1$. The question is thus to determine whether a similar phenomenon occurs \emph{on average}.

\section{Appendix}\label{se:app}

\begin{lemma}\label{le:calc}
Let $a,b,c,h$ be non-negative reals satisfying $a+\sigma b+\gamma c\geq h$ and $c\leq hp/\gamma$. Then
\[
\left(\frac{\epsilon h}{a+b+c}\right)^{a+b+c} p^{b+2c} \leq e^{-(a+b+c)},
\]
provided $p$ and $\epsilon$ are each at most absolute constants.
\end{lemma}

\begin{proof}
After rearranging and replacing $\epsilon$ by $\epsilon/e$ it is sufficient to show that
\[
f(a,b,c) := (b+2c)\log\left(\frac{1}{p}\right) - (a+b+c)\log\left(\frac{\epsilon h}{a+b+c}\right) \geq 0.
\]
We shall need the partial derivatives of $f$, which are
\begin{align*}
\frac{\partial f}{\partial a} &= 1 - \log\left(\frac{\epsilon h}{a+b+c}\right); \\
\frac{\partial f}{\partial b} &= \log\left(\frac{1}{p}\right) + 1 - \log\left(\frac{\epsilon h}{a+b+c}\right); \\
\frac{\partial f}{\partial c} &= 2\log\left(\frac{1}{p}\right) + 1 - \log\left(\frac{\epsilon h}{a+b+c}\right).
\end{align*}

Define
\[
Q = \inf \big\{ f(a,b,c) : a+\sigma b+\gamma c\geq h, \; a,b,c\geq 0 \big\}.
\]
The infimum exists because each of the partial derivatives is positive if the corresponding variable is sufficiently large, so we may restrict the domain to a compact set. Since $\partial f/\partial c > 0$ for all $c>0$, the infimum is achieved either when $a+\sigma b+\gamma c=h$ or when $c=0$.

First suppose $c=0$. In this case we have $\partial f/\partial b > 0$, so again either $a+\sigma b=h$ or $b=0$. If $b=c=0$ then $a\geq h$ and $f(a,0,0) = a\log (a/\epsilon h)$, which is non-negative if $C\leq 1$. If $c=0$ and $a+\sigma b=h$ then
\[
f(a,b,c) = f(h-\sigma b,b,0) = b\log\left(\frac{1}{p}\right) - \big(h-(\sigma-1)b\big)\log\left(\frac{\epsilon h}{h-(\sigma-1)b}\right) := g_1(b),
\]
say, and in this case
\[
Q = \inf \big\{ g_1(b) : 0\leq b\leq h/\sigma \big\}.
\]
But we have
\[
\frac{\partial g_1}{\partial b} = \log\left(\frac{1}{p}\right) + (\sigma-1)\left(\log\left(\frac{\epsilon h}{h-(\sigma-1)b}\right)-1\right) \geq \log\left(\frac{1}{p}\right)-(\sigma-1) > 0
\]
if $p<e^{-(\sigma-1)}$, so
\[
Q = \inf \big\{ g_1(b) : 0\leq b\leq h/\sigma \big\} = g_1(0) = 0,
\]
so the lemma holds in the case $c=0$.

Now suppose $a+\sigma b+\gamma c=h$. Here we have
\begin{align*}
f(a,b,c) &= f\big(h-\sigma b-\gamma c,b,c\big) \\
&= (b+2c)\log\left(\frac{1}{p}\right) - \big(h-(\sigma-1)b-(\gamma-1)c\big)\log\left(\frac{\epsilon h}{h-(\sigma-1)b-(\gamma-1)c}\right) \\
&:= g_2(b,c),
\end{align*}
say. This time,
\[
Q = \inf \big\{ g_2(b,c) : \sigma b+\gamma c\leq h, \; c\leq hp/\gamma, \; b,c\geq 0 \big\}.
\]
As with $g_1$, the partial derivative $\partial g_2/\partial b$ is strictly positive for all $b$, so
\[
Q = \inf \big\{ g_2(0,c) : 0\leq c\leq hp/\gamma \big\}.
\]
Observe that with $b=0$,
\[
\frac{\partial g_2}{\partial c} = 2\log\left(\frac{1}{p}\right) + (\gamma-1)\left(\log\left(\frac{\epsilon h}{h-(\gamma-1)c}\right)-1\right).
\]
Thus, $g_2(0,c)$ is decreasing when $c=0$, and the only zero of its partial derivative with respect to $c$ occurs at
\[
c_0 = \frac{h}{\gamma-1}\left(1-\exp\left(\frac{2\log 1/p}{\gamma-1}-1\right)\right) = \frac{h}{\gamma-1}\left(1-e^{-(1-o(1))}\right) \gg \frac{hp}{\gamma}.
\]
So $g_2(0,c)$ is decreasing on our entire range of $c$, which implies that $Q = g_2(0,hp/\gamma)$, and hence that
\[
\frac{Q}{h} = \frac{2p}{\gamma}\log\left(\frac{1}{p}\right) - \left(1-(\gamma-1)\frac{p}{\gamma}\right)\log\left(\frac{\epsilon}{1-(\gamma-1)p/\gamma}\right).
\]
Thus, $Q>0$ if $\epsilon<1$ and $p$ is sufficiently small.
\end{proof}

\section*{Acknowledgements}

P.B.~ and B.B.~ were partially supported by NSF grant DMS-1301614 and EU MULTIPLEX grant 317532, B.B.~ additionally by NSF grant DMS-0906634, and P.S.~ by a CNPq postdoctoral bolsa.

The majority of this research was carried out while the authors were visitors at the R\'enyi Institute, Budapest. It was continued while P.S.~ was a visitor at the University of Memphis, and again while all three authors were visitors at Microsoft Research, Redmond. The authors are grateful to the R\'enyi Institute and to MSR, Redmond for their kind hospitality, and P.S.~ is additionally grateful for the hospitality of the University of Memphis.

The authors would like to thank the two anonymous referees for many helpful comments and suggestions.

\bibliographystyle{amsplain}
\bibliography{bprefs}

\end{document}